\documentclass[11pt,a4paper]{amsart}

 \usepackage{amsfonts,amsmath,amscd,amssymb,amsbsy,amsthm,amstext,amsopn}
 \usepackage{fullpage,mathrsfs,subfigure}
 \usepackage[all,dvips,arc,curve,color,frame]{xy}
 \usepackage{pstricks,pst-node}

 \numberwithin{equation}{section}
\newtheorem{theorem}{Theorem}[section]
\newtheorem{proposition}[theorem]{Proposition}
\newtheorem{lemma}[theorem]{Lemma}
\newtheorem{corollary}[theorem]{Corollary}
\newtheorem{conjecture}[theorem]{Conjecture}

\theoremstyle{definition}

\newtheorem{example}[theorem]{Example}
\newtheorem{examples}[theorem]{Examples}
\theoremstyle{remark}
\newtheorem{remark}[theorem]{Remark}

\newcommand{\kk}{\ensuremath{\Bbbk}} 

\newcommand{\NN}{\ensuremath{\mathbb{N}}} 
 
\newcommand{\QQ}{\ensuremath{\mathbb{Q}}} 
 
\newcommand{\ZZ}{\ensuremath{\mathbb{Z}}} 
\newcommand{\one}{\ensuremath{(\mathrm{i})}}
\newcommand{\two}{\ensuremath{(\mathrm{ii})}}
\newcommand{\three}{\ensuremath{(\mathrm{iii})}}

\newcommand{\coh}{\operatorname{coh}}
\newcommand{\head}{\operatorname{h}}

\newcommand{\pos}{\operatorname{pos}}

\newcommand{\rank}{\operatorname{rank}}
\newcommand{\rep}{\operatorname{\mathfrak{M}_\vartheta}}   

\newcommand{\tail}{\operatorname{t}}

\newcommand{\Coker}{\operatorname{Coker}}

\newcommand{\Cox}{\operatorname{Cox}}  
\newcommand{\End}{\operatorname{End}}
\newcommand{\Ext}{\operatorname{Ext}} 
\newcommand{\Fl}{\operatorname{Fl}} 
\newcommand{\GL}{\operatorname{GL}} 
\newcommand{\Gr}{\operatorname{Gr}}  
  
\newcommand{\Hom}{\operatorname{Hom}} 
\newcommand{\Image}{\operatorname{Image}} 
 
\newcommand{\Ker}{\operatorname{Ker}} 
\newcommand{\Mat}{\operatorname{Mat}} 

\newcommand{\Pic}{\operatorname{Pic}}

\newcommand{\Rep}{\operatorname{Rep}} 
\newcommand{\Spec}{\operatorname{Spec}} 
\newcommand{\Sym}{\operatorname{Sym}} 
\newcommand{\Young}{\operatorname{Young}}

 \DeclareMathOperator{\Rderived}{\mathbf{R}\!}
 
 \newcommand{\modA}{\operatorname{mod}(\ensuremath{A})}

\newcommand{\git}{\ensuremath{/\!\!/\!}}

\title{Quiver Flag Varieties and Multigraded Linear Series}
\thanks{MSC 2000: Primary 14D22, 16G20, 18E30; Secondary 14M15, 14M25}

 \author{Alastair Craw} 
 \address{Department of Mathematics\\ University of Glasgow\\
  Glasgow\\ G12 8QW\\ United Kingdom}
 \email{anc@maths.gla.ac.uk}

\begin{document}
\bibliographystyle{plain}

 \begin{abstract}
 This paper introduces a class of smooth projective varieties that generalise and share many properties with partial flag varieties of type $A$.  The quiver flag variety $\mathcal{M}_\vartheta(Q,\underline{r})$ of a finite acyclic quiver $Q$ (with a unique source) and a dimension vector $\underline{r}$ is a  fine moduli space of stable representations of $Q$. Quiver flag varieties are Mori Dream Spaces, they are obtained via a tower of Grassmann bundles, and their bounded derived category of coherent sheaves is generated by a tilting bundle. We define the multigraded linear series of a weakly exceptional sequence of locally free sheaves $\underline{\mathscr{E}} = (\mathscr{O}_X,\mathscr{E}_1,\dots, \mathscr{E}_\rho)$ on a projective scheme $X$ to be the quiver flag variety $\vert \underline{\mathscr{E}}\vert:=\mathcal{M}_\vartheta(Q,\underline{r})$ of a pair $(Q, \underline{r})$ encoded by $\underline{\mathscr{E}}$. When each $\mathscr{E}_i$ is globally generated, we obtain a morphism $\varphi_{\vert \underline{\mathscr{E}}\vert}\colon X\to \vert \underline{\mathscr{E}}\vert$ realising each $\mathscr{E}_i$ as the pullback of a tautological bundle. As an application we introduce the multigraded Pl\"{u}cker embedding of a quiver flag variety
  \end{abstract}

 \maketitle

 \section{Introduction}
 Flag varieties of type $A$ are smooth projective varieties that come armed with a collection of tautological vector bundles which determine many algebraic invariants and geometric properties of the variety itself. For example, the Grassmannian $\Gr(\kk^n,r)$ carries one such bundle, namely the tautological quotient bundle $\mathscr{E}$ of rank $r$, whose determinant line bundle $\det(\mathscr{E})$ generates the Picard group and induces the Pl\"{u}cker embedding; on a deeper level, the celebrated results of Beilinson~\cite{Beilinson} and Kapranov~\cite{Kapranov1} establish that  the direct sum of certain Schur powers of $\mathscr{E}$ provides a tilting generator for the bounded derived category of coherent sheaves on $\Gr(\kk^n,r)$. These results were extended to partial flag varieties of type $A$ by Kapranov~\cite{Kapranov2}, where the description of the flag variety as a tower of Grassmann-bundles was exploited. 
 
 The primary goal of this paper is to extend these results to a much broader class of varieties. Let $Q$ be a finite acyclic quiver with a unique source, and let $\underline{r}$ be a dimension vector for which the component of $\underline{r}$ corresponding to the source vertex is 1. The quiver flag variety  $\mathcal{M}_\vartheta(Q,\underline{r})$ is defined to be the fine moduli space $\mathcal{M}_\vartheta(Q,\underline{r})$ of $\vartheta$-stable representations of $Q$ with dimension vector $\underline{r}$, where $\vartheta$ is a given stability condition. Familiar examples include partial flag varieties of type $A$,  products of Grassmannians,  and certain toric varieties studied by Hille~\cite{Hille}; the empty set can occur, but we provide a simple necessary and sufficient condition for nonempty moduli. Our guiding philosophy is that quiver flag varieties are similar in spirit to flag varieties, and we provide some evidence for this by establishing the following geometric result.
 
 \begin{theorem}
  \label{thm:towerintro}
 For any quiver $Q$ and vector $\underline{r}$ as above,   $\mathcal{M}_\vartheta(Q,\underline{r})$ is a smooth projective Mori Dream Space that gives rise to a tower of Grassmann-bundles
 \begin{equation}
 \label{eqn:towerintro}
 \mathcal{M}_\vartheta(Q,\underline{r})=Y_\rho\longrightarrow Y_{\rho-1}\longrightarrow
 \cdots \longrightarrow Y_1\longrightarrow Y_0=\Spec \kk.
 \end{equation}
 where each $Y_{j}$ is a quiver flag variety, and where $Q_0=\{0,1,\dots,\rho\}$ is the vertex set of $Q$.
 \end{theorem}

  \noindent Since quiver flag varieties are Mori Dream Spaces (see Hu--Keel~\cite{HuKeel}), all flips, flops, divisorial contractions and Mori fibre space structures of $\mathcal{M}_\vartheta(Q,\underline{r})$ are obtained by variation of GIT quotient. Put another way, the given moduli construction of a quiver flag variety lends itself well to understanding the birational geometry of  $\mathcal{M}_\vartheta(Q,\underline{r})$.
 
 \medskip
 
 Quiver flag varieties also share algebraic properties with partial flag varieties of type $A$.  For the strongest statement of this kind,  recall that for any locally free sheaf $\mathscr{E}$, the Schur powers of $\mathscr{E}$ are locally free sheaves $\mathbb{S}^\lambda\mathscr{E}$ associated to Young diagrams $\lambda$; examples include the symmetric powers and the alternating products of $\mathscr{E}$. When the fine moduli space $\mathcal{M}_\vartheta(Q,\underline{r})$ is nonempty it carries a collection of tautological bundles $\mathscr{W}_0,\mathscr{W}_1,\dots,\mathscr{W}_\rho$ indexed by the vertex set of $Q$, where the rank of $\mathscr{W}_i$ is equal to the $i$th component $r_i$ of $\underline{r}$, and where $\mathscr{W}_0$ is the trivial bundle. We define integers $s_1,\dots, s_\rho$ depending on both $Q$ and $\underline{r}$ that satisfy $s_i> r_i$ for all $i\geq 1$. Let  $\Young(s_i-r_i,r_i)$ denote the set of Young diagrams with at most $s_i-r_i$ columns and at most $r_i$ rows. 

\begin{theorem}
\label{thm:quivertiltintro}
For any quiver flag variety $\mathcal{M}_\vartheta(Q,\underline{r})$, the locally free sheaf
\[
\mathscr{T}:=\bigoplus_{1\leq i\leq \rho}\bigoplus_{\lambda_i\in \Young(s_i-r_i,r_i)} \mathbb{S}^{\lambda_1}\mathscr{W}_1\otimes \dots \otimes \mathbb{S}^{\lambda_\rho}\mathscr{W}_\rho
\]
is a tilting bundle on $\mathcal{M}_\vartheta(Q,\underline{r})$. In particular, if we write $\modA$ for the abelian category of finitely generated right modules over $\End(\mathscr{T})$, then the functor
 \[
 \Rderived\Hom(\mathscr{T},-)\colon
 D^b\big(\!\coh(\mathcal{M}_\vartheta(Q,\underline{r}))\big)\longrightarrow D^b\big(\!\modA\big)
 \]
 is an exact equivalence of triangulated categories.
 \end{theorem}

 The proof relies on a vanishing theorem for appropriate Schur powers of the tautological bundles (see Proposition~\ref{prop:quivervanishing}). The special case where the quiver is a chain with only one arrow between adjacent vertices recovers Kapranov's result ~\cite{Kapranov2} for partial flag varieties of type $A$.
 
 It is worth making a few remarks about this result. Firstly, quiver flag varieties are not Fano in general, so Theorem~\ref{thm:quivertiltintro} provides a large class of varieties with explicit tilting bundles that are nevertheless not Fano. Secondly, the moduli construction gives $\mathcal{M}_\vartheta(Q,\underline{r})$ as the geometric quotient of a Zariski-open subset of a vector space by the linear action of a reductive group, so Theorem~\ref{thm:quivertiltintro} provides further evidence for the following conjecture of King~\cite{King2}:
 
\begin{conjecture}
\label{conj:King}
Any smooth complete variety obtained as the geometric quotient of a Zariski-open subset of a vector space by the linear action of a reductive group has a tilting bundle.
\end{conjecture}

\noindent Finally, in the special case when $\mathcal{M}_\vartheta(Q,\underline{r})$ is a toric  quiver variety, the tilting bundle $\mathscr{T}$ is simply the direct sum of the line bundles 
 \begin{equation}
 \label{eqn:torictiltintro}
 \big\{\mathscr{W}_1^{\theta_1}\otimes \dots \otimes \mathscr{W}_\rho^{\theta_\rho} : 0\leq \theta_i < s_i \text{ for } 1\leq i\leq \rho\big\},
 \end{equation}
 and the vanishing result is optimal (see Remark~\ref{rem:toricvanishing}). In the special case when $\mathcal{M}_\vartheta(Q,\underline{r})$ is a Fano toric variety, Altmann--Hille~\cite{AltmannHille} showed that the tautological bundles $(\mathscr{W}_0,\dots,\mathscr{W}_\rho)$ form a strongly exceptional sequence. Theorem~\ref{thm:quivertiltintro}, and more specifically, the set of line bundles from \eqref{eqn:torictiltintro}, extends this sequence to a full, strongly exceptional sequence. While we do not need the Fano assumption,  our quivers are required to have a unique source.

 \medskip
  
The secondary goal of this paper is to demonstrate that quiver flag varieties arise naturally as ambient spaces in algebraic geometry,  generalising the classical linear series of a line bundle or, more generally, the Grassmannian associated to a vector bundle of higher rank. We define a sequence of distinct locally free sheaves $\underline{\mathscr{E}} = (\mathscr{O}_X,\mathscr{E}_1,\dots, \mathscr{E}_\rho)$ on a projective scheme $X$ to be `weakly exceptional' if $H^0(\mathscr{E}_i)\neq 0$ for $i>0$ and $\Hom(\mathscr{E}_j,\mathscr{E}_i)=0$ for $j>i$.  Any such collection defines a finite, acyclic quiver $Q$ with a unique source; the inclusion of $\mathscr{O}_X$ forces $Q$ to have a unique source and provides a geometric interpretation of the framing of a quiver, see Remark~\ref{rem:framings2}.  The multigraded linear series is defined to be the quiver flag variety $\vert \underline{\mathscr{E}}\vert:=\mathcal{M}_\vartheta(Q,\underline{r})$ for the dimension vector $\underline{r}=(1,\rank(\mathscr{E}_1),\dots, \rank(\mathscr{E}_\rho))$.  Examples include the Grassmannian of rank $r$ quotients of a locally free sheaf $\mathscr{E}$, and the multilinear series of a sequence of effective line bundles on a projective toric variety from Craw--Smith~\cite{CrawSmith}.

 \begin{theorem}
 \label{thm:morphismintro}
 Let $\underline{\mathscr{E}}=(\mathscr{O}_X,\mathscr{E}_1,\dots, \mathscr{E}_\rho)$ be a weakly exceptional sequence of globally generated locally free sheaves on a projective scheme $X$.  There is a morphism $\varphi_{\vert\underline{\mathscr{E}}\vert} \colon X\rightarrow
 \vert\underline{\mathscr{E}}\vert$ that satisfies $\varphi_{\vert\underline{\mathscr{E}}\vert}^*(\mathscr{W}_i)=\mathscr{E}_i$ for all $0\leq
 i\leq \rho$, where $\mathscr{W}_0,\dots.\mathscr{W}_\rho$ are the tautological bundles on $\vert\underline{\mathscr{E}}\vert$
 \end{theorem}

As an application, we define the multigraded Pl\"{u}cker embedding of a quiver flag variety $\mathcal{M}_\vartheta$ to be the morphism $\varphi_{\vert\underline{\det}(\mathscr{W})\vert}$ determined by the sequence  $\underline{\det}(\mathscr{W}) = \big(\mathscr{O}_{\mathcal{M}_\vartheta},\det(\mathscr{W}_1),\dots,\det(\mathscr{W}_\rho)\big)$; this is a closed immersion, and generalises the classical morphism for the Grassmannian.  We conjecture that  the Cox ring of $\mathcal{M}_\vartheta$ is obtained as a quotient of the Cox ring of $\vert\underline{\det}(\mathscr{W})\vert$. Put geometrically, the multigraded Pl\"{u}cker embedding $\varphi_{\vert\underline{\det}(\mathscr{W})\vert}$ is such that all morphisms obtained by running the Mori programme on $\mathcal{M}_\vartheta$ are obtained by restriction from those of $\vert\underline{\det}(\mathscr{W})\vert$. It would be interesting to find a nice presentation for the Cox ring, and hence to compute explicitly the ideal that cuts out the image of the multigraded Pl\"{u}cker embedding.

 We now describe the structure of the paper. Section~\ref{sec:two} gives the moduli construction of quiver flag varieties and establishes some of their basic geometric properties. The description as Mori Dream Spaces and the structure theorem for quiver flag varieties, which together provide the ingredients for Theorem~\ref{thm:towerintro}, are presented in Section~\ref{sec:three}. The derived category results, including the key vanishing result and the construction of the tilting bundle from Theorem~\ref{thm:quivertiltintro}, are proven in Section~\ref{sec:four}.  Finally, in Section~\ref{sec:five} we introduce the multigraded linear series and prove Theorem~\ref{thm:morphismintro}, concluding with the construction of the multigraded Pl\"{u}cker embedding. 
 
\medskip

 \noindent \textbf{Notation and conventions:} We identify throughout a vector bundle and its locally free sheaf of sections. For a locally free sheaf $\mathscr{F}$
 on a scheme $Y$, and for a positive integer $r< \rank(\mathscr{F})$, we
 let $\Gr(\mathscr{F},r)$ denote the Grassmannian parameterising
 quotients of $\mathscr{F}$ that are locally free of rank $r$. As a
 special case we write $\mathbb{P}^*(\mathscr{F}):= \Gr(\mathscr{F},1)$. For a locally free sheaf $\mathscr{E}$ or $\mathscr{W}$ on a scheme $X$, its fibre over a given closed point $x\in X$ is typically denoted $E$ or $W$ respectively. For $c,r\in \NN$, let $\Mat(r\times c,\kk)$ denote the space of $r\times c$ matrices with coefficients in $\kk$. Given a finite subset $S$ of a lattice $\ZZ^d$, write $\pos(S)=\big{\{}\sum_{s\in S} \lambda_s s \in \QQ^d : \lambda_s\geq 0\big{\}}$ for the rational cone generated by $S$.
 
  \medskip

 \noindent \textbf{Acknowledgements.} This article was written during
 the programme in algebraic geometry at MSRI in 2009, and I am grateful to the
 organisers for establishing a stimulating atmosphere. I owe special
 thanks to Shigeru Mukai whose talk at MSRI provided the catalyst for this
 project, and to Valery Alexeev for a very helpful discussion. Thanks also to Ana-Maria Castravet, Rob Lazarsfeld, Nick Proudfoot, Greg Smith, and to the anonymous referee for making several helpful comments. The author was supported in part by EPSRC.
 
 \section{Quiver flag varieties}
  \label{sec:two}
 We begin by constructing quiver flag varieties as fine moduli spaces of quiver representations. Examples include partial flag varieties of type $A$, products of Grassmannians, and certain toric quiver varieties. We work over an algebraically closed field $\kk$.
 
 Let $Q$ be a finite, acyclic quiver with a unique source. The quiver
 is encoded by a finite set of vertices $Q_0$, a finite set of arrows
 $Q_1$, and a pair of maps $\head, \tail\colon Q_1\to Q_0$ indicating
 the head and tail of each arrow. Write the vertex set as
 $Q_0=\{0,1,\dots, \rho\}$ where $0$ is the source.  Since $Q$ is
 acyclic, we may assume that for each pair $i, j\in Q_0$, the number
 $n_{i,j}$ of arrows with tail at $i$ and head at $j$ is zero whenever
 $i>j$. Let $\kk Q$ denote the path algebra of $Q$. Each vertex $i\in Q_0$ determines an idempotent $e_i\in \kk Q$, and the subspace $e_i(\kk Q)e_j$ of $\kk Q$ is spanned by paths with tail at $i\in Q_0$ and head at $j\in Q_0$.
 
 A representation of the quiver $Q$ consists of a $\kk$-vector space
 $W_i$ for $i\in Q_0$ and a $\kk$-linear map $w_a \colon
 W_{\tail(a)} \to W_{\head(a)}$ for $a \in Q_1$.  We often write
 $W$ as shorthand for $\big((W_i)_{i\in Q_0},(w_a)_{a\in Q_1}\big)$.
 The dimension vector of $W$ is the vector $\underline{r}\in
 \ZZ^{\rho+1}$ with components $r_i = \dim_\kk (W_i)$ for $0\leq i\leq
 \rho$. A map of representations $\psi\colon W\to W^\prime$ is a
 family $\psi_{i} \colon W_i^{\,} \to W_i^\prime$ of $\kk$-linear maps
 for $0\leq i\leq \rho$ satisfying $w_a^\prime \psi_{\tail(a)} =
 \psi_{\head(a)} w_a$ for $a \in Q_1$.  With composition defined
 componentwise, we obtain the abelian category of finite dimensional
 representations of $Q$. To define the appropriate stability notion,
 consider $\theta \in \QQ^{\rho+1}$ and associate to each
 representation $W$ the rational number $\theta(W) := \sum_{0\leq
   i\leq \rho} \theta_i\dim_\kk(W_i)$.  Following King~\cite{King}, we
 say that $W$ is $\theta$-semistable if $\theta(W)=0$ and every
 subrepresentation $W^\prime \subset W$ satisfies
 $\theta(W^\prime)\geq 0$. Moreover, $W$ is $\theta$-stable if the
 only subrepresentations $W^\prime$ with $\theta(W^\prime)=0$ are 0
 and $W$.

 Once and for all, set $r_0=1$, choose positive integers $r_1,\dots,
 r_\rho$ and set $\underline{r}=(1,r_1,\dots, r_\rho)$. The space of
 representations of $Q$ of dimension vector $\underline{r}$ is the
 $\kk$-vector space
 \begin{equation}
 \label{eqn:finegrading}
 \Rep(Q,\underline{r}) = \bigoplus\limits_{a \in Q_1} \Hom_{\kk}\big(\kk^{r_{\tail(a)}}, \kk^{r_{\head(a)}}\big).
 \end{equation}
 Isomorphism classes of representations correspond one-to-one to
 orbits in $\Rep(Q,\underline{r})$ under the action of the group
 $\prod_{0\leq i\leq \rho} \GL(r_i)$ induced by change of basis;
 explicitly, for a group element $g=(g_i)$ and a point $w = (w_a)$ we have $(g\cdot w)_a =
 g_{\head(a)}w_ag_{\tail(a)}^{-1}$. The diagonal one-dimensional torus
 $\kk^\times$ acts trivially, leaving a faithful action of the
 quotient group
 \[
 G= G(\underline{r}):=\Big( \prod_{0\leq i\leq \rho}
 \GL(r_i)\Big)/\kk^\times
 \]
 on $\Rep(Q,\underline{r})$.  Write $G^\vee = \{\theta = (\theta_i)\in
 \ZZ^{\rho+1} : \sum_{0\leq i\leq \rho} \theta_i r_i = 0\}$ for the group of characters of $G$. For any character $\theta\in G^\vee$,
 King~\cite[Proposition~3.1]{King} proved that a point $(w_a)\in
 \Rep(Q,\underline{r})$ is $\theta$-stable in the sense of Geometric
 Invariant Theory if and only if the corresponding representation $W=
 \big((W_i)_{i\in Q_0},(w_a)_{a\in Q_1}\big)$ is $\theta$-stable in
 the sense defined above. 
 
 It is often convenient to consider a coarser decomposition of
 $\Rep(Q,\underline{r})$ than that in \eqref{eqn:finegrading}.  For
 $1\leq i\leq \rho$ we define the subspace $\Rep(Q,\underline{r})_i:=
 \bigoplus_{\{a \in Q_1: \: \head(a)=i\}}
 \Hom_{\kk}\big(\kk^{r_{\tail(a)}}, \kk^{r_i}\big)$, giving
 \begin{equation}
 \label{eqn:coarsegrading}
 \Rep(Q,\underline{r}) = \bigoplus\limits_{1\leq i\leq \rho}
 \Rep(Q,\underline{r})_i.
 \end{equation}
 In decomposition \eqref{eqn:finegrading} we write points of
 $\Rep(Q,\underline{r})$ as $(w_a)$ where each
 $w_a$ is a matrix with $r_{\head(a)}$ rows, and in decomposition
 \eqref{eqn:coarsegrading} we write $(w_i)$ where each
 $w_i$ is a matrix with $r_i$ rows.

 The key to the construction of quiver flag varieties is the choice of the \emph{special weight vector}
 \[\vartheta:= \Big(-\sum_{i=1}^{\rho} r_i, 1, 1, \dots,
   1\Big)\in G^\vee.
   \]
 \begin{lemma}
 \label{lem:generic}
  The following are equivalent for the special weight vector $\vartheta\in G^\vee$:
 \begin{enumerate}
 \item[\one] the representation $(w_i)\in \Rep(Q,\underline{r})$ is $\vartheta$-stable;
 \item[\two]  the representation $(w_i)\in \Rep(Q,\underline{r})$ is $\vartheta$-semistable;
 \item[\three] for each $1\leq i\leq \rho$, the matrix $w_i$ has full rank,
   that is, $\rank(w_i)=r_i$.
 \end{enumerate}
 \end{lemma}
 \begin{proof}
   Let $W$ denote the representation of $Q$ corresponding to the point $(w_i)\in \Rep(Q,\underline{r})$. 
   Since $r_0=1$ we have $\vartheta(W) = 0$ and implication
   \one$\implies$\two\ follows. For \two$\implies$\three, suppose for
   a contradiction that $W$ is $\vartheta$-semistable and
   $\rank(w_i)<r_i$ for some $i$.  For $j\in Q_0$, define a
   $\kk$-vector space $W^\prime_j$ by setting $W^\prime_j:=W_j$ for
   $j\neq i$, and $W^\prime_i:= \Image(w_i)$. In addition, for $a\in
   Q_1$, define $\kk$-linear maps $w^\prime_a\colon
   W_{\tail(a)}^\prime\to W_{\head(a)}^\prime$ by setting $w_a^\prime
   = w_a$ for $\tail(a)\neq i$; when $\tail(a)=i$, set $w_a^\prime$ to
   be the restriction of $w_a\colon W_i\to W_{\head(a)}$ to
   $\Image(w_i)\subset W_i$. The resulting subrepresentation
   $W^\prime$ of $W$ is proper since $W_i^\prime\neq W_i$ and,
   moreover,
 \[
 \vartheta(W^\prime) = \vartheta(W) - \vartheta_i
 \dim_\kk(\Coker(w_i)) = -\dim_\kk(\Coker(w_i)) < 0
 \]
 which is a contradiction. For \three$\implies$\one, suppose each
 matrix $w_i$ has full rank and let $W^\prime$ be a nonzero
 subrepresentation of $W$ satisfying $\vartheta(W^\prime)\leq 0$. To
 prove that $W$ is $\vartheta$-stable we show that $W^\prime=W$. The
 only nonpositive entry of $\vartheta$ is $\vartheta_0$, so
 $\dim(W_0^\prime)>0$ and hence $W_0^\prime = W_0$. Suppose 
 $W^\prime_i=W_i$ for $i<j$.  Since $W^\prime$ is a subrepresentation
 of $W$, the $\kk$-linear maps satisfy $w^\prime_a = w_a \colon
 W_{\tail(a)}\to W_{\head(a)}$ for $\tail(a)<j$. This gives
 $(w_j^\prime) = (w_j)$ in decomposition \eqref{eqn:coarsegrading}.
 The fact that $\rank(w_j)=r_j$ forces $W^\prime_j=W_j$. The result follows
 by induction on $j$.
\end{proof}

  It follows that a $\vartheta$-stable representation of $Q$ of
  dimension vector $\underline{r}$ exists when 
  \begin{equation}
 \label{eqn:numerical}
 r_i \leq s_i:= \sum_{\{a\in Q_1 :\;
  \head(a)=i\}} r_{\tail(a)}\quad\quad \text{for all }1\leq i\leq \rho.
 \end{equation}
 The subset of $\vartheta$-stable points is open in the affine space
 $\Rep(Q,\underline{r})$, so if the numerical condition
 \eqref{eqn:numerical} is satisfied then the locus of
 $\vartheta$-stable points in $\Rep(Q,\underline{r})$ is dense. To
 describe explicitly the ideal that cuts out the $\vartheta$-unstable
 locus we work with decomposition \eqref{eqn:coarsegrading}. For
 $1\leq i\leq \rho$, identify $\Rep(Q,\underline{r})_i$ with the space
 of matrices $\Mat(r_i\times s_i, \kk)$.  For any subset $I\subset
 \{1,\dots, s_i\}$ with $\vert I\vert = r_i$, let $m_I(i)\in
 \kk[\Mat(r_i\times s_i, \kk)]$ denote the polynomial of degree $r_i$
 that computes the minor of an $r_i\times s_i$-matrix whose columns
 are indexed by the entries of $I$. For any $1\leq i\leq \rho$, the
 polynomial ring $\kk[\Rep(Q,\underline{r})]$ is obtained from
 $\kk[\Rep(Q,\underline{r})_i]$ by adding extra variables, so each
 $m_I(i)$ is naturally a polynomial of degree $r_i$ in 
 $\kk[\Rep(Q,\underline{r})]$.
 
 \begin{proposition}
 \label{prop:finemoduli}
 Let $Q$ be a finite acyclic quiver with a unique source $0\in Q_0$, and suppose
that $\underline{r}=(1,r_1,\dots, r_\rho)$ satisfies
 \eqref{eqn:numerical}. For the special weight vector $\vartheta$, the following spaces coincide:
 \begin{enumerate}
 \item[\one] the GIT quotient $\Rep(Q,\underline{r})\git_\vartheta G$
   linearised by $\vartheta\in G^\vee$;
 \item[\two] the geometric quotient of $\Rep(Q,\underline{r})\setminus
   \mathbb{V}(B)$ by the action of the reductive group $G$, where \[ B
   = \bigcap_{1\leq i\leq \rho} \Big(m_I(i) \in
   \kk[\Rep(Q,\underline{r})] : I\subset \{1,\dots, s_i\} \text{ with
   }\vert I\vert = r_i\Big)
 \]
 is the irrelevant ideal;
 \item[\three] the fine moduli space
  $\mathcal{M}_\vartheta(Q,\underline{r})$ of $\vartheta$-stable
  representations of $Q$ of dimension vector $\underline{r}$.
 \end{enumerate}
 Moreover, this space is a smooth projective variety of dimension
 $\sum_{1\leq i\leq \rho} r_i(s_i-r_i)$.
 \end{proposition}

 \begin{proof}
   The variety $\Rep(Q,\underline{r})\git_\vartheta G$ is defined to
   be the categorical quotient of the open subscheme of
   $\vartheta$-semistable points in $\Rep(Q,\underline{r})$ by the
   action of $G$.  Assumption \eqref{eqn:numerical} ensures that this
   variety is nonempty, and Lemma~\ref{lem:generic} shows that
   $\Rep(Q,\underline{r})\git_\vartheta G$ coincides with the
   geometric quotient of the open subscheme of $\vartheta$-stable
   points by the action of $G$.  Lemma~\ref{lem:generic} also shows
   that the $\vartheta$-unstable locus comprises points $(w_i)\in
   \Rep(Q,\underline{r})$ satisfying $\rank(w_i)<r_i$ for some $1\leq
   i\leq \rho$. The vanishing of all minors $m_I(i)$ detects this drop
   in rank, so a point lies in $\mathbb{V}(B)$ if and only if it is
   $\vartheta$-unstable. This shows that \one\ and \two\ coincide. The
   dimension vector $\underline{r}$ is indivisible since $r_0=1$, and
   every $\vartheta$-semistable representation of $Q$ is
   $\vartheta$-stable by Lemma~\ref{lem:generic}, so varieties \one\ 
   and \three\ coincide by King~\cite[Proposition~5.3]{King}.  As for
   smoothness, move a $\vartheta$-stable point $(w_i)\in
   \Rep(Q,\underline{r})$ in its $G$-orbit so that each full rank
   matrix $w_i\in \Mat(r_i\times s_i, \kk)$ is in reduced echelon
   form. The resulting matrix has $r_i(s_i-r_i)$ free entries,
   producing a cover of $\mathcal{M}_\vartheta(Q,\underline{r})$ by
   affine spaces of dimension $\sum_{1\leq i\leq \rho} r_i(s_i-r_i)$.
   Projectivity follows since $\Rep(Q,\underline{r})\git_\vartheta G$
   is projective over $\Spec\kk[\Rep(Q,\underline{r})]^G = \Spec \kk$.
 \end{proof}
 
 \begin{remark}
  \label{rem:strict}
 The proof shows that there is some redundancy in the construction when $r_i = s_i$ for some $i\geq 1$ (compare also Theorem~\ref{thm:tower} to follow). It is often convenient to assume, therefore, that if the numerical condition \eqref{eqn:numerical} holds for all $i\geq 1$, then it holds strictly for all $i\geq 1$. This assumption is familiar for Grassmannians and partial flag varieties, see Examples~\ref{exa:Grassmannian} and \ref{exa:flag}.
 \end{remark}
   
 We call
 $\mathcal{M}_\vartheta:=\mathcal{M}_\vartheta(Q,\underline{r})$ the
 \emph{quiver flag variety} of the quiver $Q$ and dimension vector
 $\underline{r}$. As Proposition~\ref{prop:finemoduli}\three\ states,
 this variety is the fine moduli space of $\vartheta$-stable
 representations of $Q$ of dimension vector $\underline{r}$.
 Explicitly, $\mathcal{M}_\vartheta$ represents the functor
 $\rep(Q,\underline{r})$ that assigns to each connected scheme $B$ the
 set of locally free sheaves on $B$ of the form
 \begin{equation}
 \label{eqn:functor}
 \rep(Q,\underline{r})(B):= \left\{\bigoplus_{i \in Q_0} \mathscr{W}_i
   : \begin{array}{c} \rank(\mathscr{W}_i)=r_i \;\;\forall\; i\in Q_0,
     \text{ with }\mathscr{W}_0\cong \mathscr{O}_B\\ \exists \text{
       morphism }\mathscr{W}_{\tail(a)}\to
     \mathscr{W}_{\head(a)}\;\;\forall\; a\in Q_1 \\ \text{each fibre
     }\bigoplus_{i\in Q_0}W_i \text{ is
     }\vartheta\text{-stable}\end{array}\right\}/\sim_{\text{isom}}.
 \end{equation}
 In particular, the quiver flag variety
 $\mathcal{M}_\vartheta$ carries
 \begin{enumerate}
 \item[\one] a universal locally free sheaf $\bigoplus_{i \in Q_0}
   \mathscr{W}_i$ with $\rank(\mathscr{W}_i)=r_i$ and
   $\mathscr{W}_0\cong
   \mathscr{O}_{\mathcal{M}_\vartheta}$; and
 \item[\two] a universal $\kk$-algebra homomorphism 
  \[
  \kk Q\to
 \End\Big(\bigoplus_{i \in Q_0} \mathscr{W}_i\Big)
  \]
  obtained by
 composing morphisms $\mathscr{W}_{\tail(a)}\to \mathscr{W}_{\head(a)}$
 arising from arrows $a\in Q_1$ to associate morphisms to paths.
 \end{enumerate}
 As King~\cite{King} notes, the construction of moduli spaces of
 quiver representations has an inherant ambiguity in the choice of
 isomorphism $G\cong \prod_{1\leq i\leq \rho} \GL(r_i)$; our choice
 gives $\mathscr{W}_0\cong \mathscr{O}_{\mathcal{M}_\vartheta}$.

 \begin{example}
 \label{exa:Grassmannian}
 For $\rho=1$, write $n$ for the number of arrows from 0 to 1
 and set $r:=r_1$. Then $Q$ is the Kronecker quiver and $\Rep(Q,\underline{r})=\Mat(r\times
 n,\kk)$. The projection $\GL(1)\times \GL(r)\to \GL(r)$
 identifies $G$ with $\GL(r)$, after which the $G$-action on
 $\Rep(Q,\underline{r})$ coincides with the left $\GL(r)$-action on
 $\Mat(r\times n,\kk)$. Lemma~\ref{lem:generic} shows that
 $\vartheta$-stable points in $\Rep(Q,\underline{r})$ are matrices of rank $r$. If $r>n$ then the $\vartheta$-stable locus is
 empty, hence so is $\mathcal{M}_\vartheta$.
 Otherwise, Proposition~\ref{prop:finemoduli} shows that
 $\mathcal{M}_\vartheta$ parametrises $\GL(r)$-orbits
 of full rank matrices in $\Mat(r\times n,\kk)$, that is, surjective
 maps $\kk^n\twoheadrightarrow E_1$ for which $\dim(E_1)=r$. Thus,
 $\mathcal{M}_\vartheta$ is isomorphic to the
 Grassmannian $\Gr(\kk^n,r)$ of $r$-dimensional quotients of $\kk^n$, and the tautological bundle
 $\mathscr{W}_1$ is isomorphic to the rank $r$ quotient bundle $\mathscr{E}_1$ on
 $\Gr(\kk^n,r)$. 
 \end{example}

 \begin{example}
  More generally, let $Q$ denote the quiver with $Q_0=\{0,1,\dots, \rho\}$, where for
   each $1\leq i\leq \rho$ there are $n_i:= n_{0,i}\geq 2$ arrows from 0 to
   $i$, and where $n_{i,j}=0$ otherwise. If $\underline{r}$ satisfies $r_i<n_i$ for $i>0$, then $\mathcal{M}_\vartheta$ is isomorphic to
   the product of Grassmannians $\prod_{1\leq i\leq \rho} \Gr(\kk^{n_{i}},
   r_i)$.  
 \end{example}

 \begin{example}
 \label{exa:flag}
 Let $Q$ be the quiver with $n:=n_{0,1}$, $n_{i,i+1}=1$ for $1\leq
 i\leq \rho-1$, and $n_{i,j}=0$ otherwise; see Figure~\ref{fig:flag}(a).
Condition \eqref{eqn:numerical} is strict when
 $n>r_1>r_2>\dots > r_\rho$. Lemma~\ref{lem:generic} identifies
 $\vartheta$-stable points of $\Rep(Q,\underline{r})$ with
 chains $[\kk^n\twoheadrightarrow \kk^{r_1}\twoheadrightarrow \dots \twoheadrightarrow \kk^{r_\rho}]$ of surjective
 $\kk$-linear maps, and the action of $G$ is change of basis on the vector spaces $\kk^{r_i}$ in the chain. Thus, 
 isomorphism classes of $\vartheta$-stable points of
 $\Rep(Q,\underline{r})$ are quotient flags $[\kk^n\twoheadrightarrow
 E_1\twoheadrightarrow \dots \twoheadrightarrow E_\rho]$ where
 $\dim_\kk(E_i)=r_i$ for $1\leq i\leq \rho$. In particular,
 $\mathcal{M}_\vartheta$ is isomorphic to the partial flag variety
  \[
 \Fl(n;r_1,\dots,r_\rho):= \Big\{\text{Quotient flags }\kk^n\twoheadrightarrow
 E_1\twoheadrightarrow \dots \twoheadrightarrow E_\rho \; : \dim_\kk(E_i)=r_i \text{ for }1\leq
 i\leq \rho\Big\},
 \]
 and for each $0\leq i\leq \rho$, the tautological bundle
 $\mathscr{W}_i$ is isomorphic to the quotient bundle $\mathscr{E}_i$ on $\Fl(n;r_1,\dots,r_\rho)$ whose fibre at a point is the $\kk$-vector space $E_i$ in
 the corresponding flag.
  \end{example}

  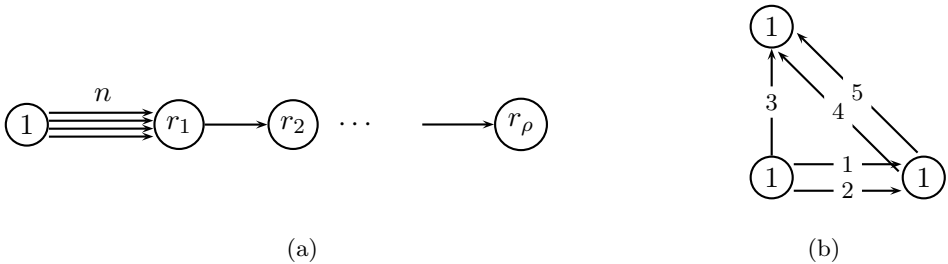
\begin{figure}[!ht]
    \centering
  \subfigure[]{
        \psset{unit=1cm}
        \begin{pspicture}(1,-0.2)(6.5,2)
          \cnodeput(0,1){A}{1}
          \cnodeput(2,1){B}{$r_1$} 
          \cnodeput(3.5,1){C}{$r_2$}
         \pnode(5.2,1){E}{}
          \cnodeput(6.5,1){F}{$r_\rho$}
          \psset{nodesep=0pt}
          \ncline[offset=4.5pt]{->}{A}{B}
          \ncline[offset=-4.5pt]{->}{A}{B}
           \ncline[offset=1.5pt]{->}{A}{B}
          \ncline[offset=-1.5pt]{->}{A}{B}
          \rput(1,1.4){$n$}
          \ncline{->}{B}{C}
          \ncline{->}{E}{F}
          \rput(4.35,1){$\cdots$}
      \end{pspicture}}\hspace{1in}
      \subfigure[]{
        \psset{unit=1cm}
        \begin{pspicture}(0,-0.5)(2.5,1.9)
          \cnodeput(0.5,0){A}{1}
          \cnodeput(2.5,0){B}{1} 
          \cnodeput(0.5,2){C}{1}
          \psset{nodesep=0pt}
          \ncline[offset=5pt]{->}{A}{B}\lput*{:U}{\footnotesize{1}}
          \ncline[offset=-5pt]{->}{A}{B}\lput*{:U}{\footnotesize{2}}
          \ncline{->}{A}{C} \lput*{:270}{\footnotesize{3}}
          \ncline[offset=5pt]{<-}{C}{B}\lput*{:45}{\footnotesize{5}}
          \ncline[offset=5pt]{->}{B}{C}\lput*{:225}{\footnotesize{4}}
        \end{pspicture}
       }
    \caption{Quivers defining (a) a flag variety; and (b) a
      $\mathbb{P}^2$-bundle over $\mathbb{P}^1$. In each case, the
      numbers at vertices are the components of the dimension vector \label{fig:flag}}
  \end{figure}
 
 If the dimension vector is $\underline{r} = (1,\dots,1)\in \ZZ^{\rho+1}$, then $G$ is an algebraic torus and $\mathcal{M}_\vartheta$ is the toric variety with Cox ring $\kk[y_a : a\in Q_1]$ studied by Craw--Smith~\cite[Proposition~3.8]{CrawSmith}; this is a toric quiver variety as defined by Hille~\cite{Hille}. To emphasis our choice of the special weight vector $\vartheta$, we refer to $\mathcal{M}_\vartheta$ as a \emph{toric quiver flag variety}.  In this case, the quotient construction from Proposition~\ref{prop:finemoduli}\two\ coincides
 with the Cox construction of the toric variety.

 \begin{example}
 \label{exa:toric}
 For the quiver $Q$ from Figure~\ref{fig:flag}(b) and for $\vartheta =
 (-2,1,1)$, the $\vartheta$-unstable locus in
 $\Rep(Q,\underline{r})=\Spec\kk[y_1,\dots, y_5]$ is cut out by $B = (y_1,y_2) \cap (y_3, y_4, y_5)$, and the incidence matrix
 of $Q$ defines a $\ZZ^2$-grading of $\kk[y_1,\dots, y_5]$ that
 induces the $G$-action on $\Rep(Q,\underline{r})$. We will see
 in Example~\ref{exa:toric2} that $\mathcal{M}_\vartheta =
 \mathbb{P}_{\mathbb{P}^1}\big(\mathscr{O}\oplus \mathscr{O}(1)\oplus
   \mathscr{O}(1)\big)$.
 \end{example}

 \begin{remark}
 \label{rem:Ginzburg}
 Crawley--Boevey~\cite{CrawleyBoevey} observed that representations of
 a quiver $Q$ with a unique source $0$ of dimension vector
 $\underline{r}$ with $r_0=1$ can be interpreted as certain framed
 representations of a quiver $\overline{Q}$ obtained from $Q$ by
 removing the vertex $0$ and all arrows with tail at 0.
 Ginzburg~\cite[\S3]{Ginzburg} provides a nice explanation and presents analogues of our Lemma~\ref{lem:generic},
 Proposition~\ref{prop:finemoduli} and Example~\ref{exa:flag}. We
 provide a purely geometric interpretation of the framing in
 Section~\ref{sec:five}.
 \end{remark}

 \section{Birational geometry} 
 \label{sec:three}
This section  investigates how quiver flag varieties change under variation of GIT quotient and presents a structure theorem that describes each via a tower of Grassmann-bundles. Throughout this section, $Q$ is a finite, acyclic quiver with a
 unique source $0\in Q_0$, the dimension vector
 $\underline{r}=(1,r_1,\dots,r_\rho)$ satisfies \eqref{eqn:numerical} and
 $\mathcal{M}_\vartheta=\Rep(Q,\underline{r})\git_\vartheta G$ is the
 quiver flag variety.
  
  It is well known that any flip, flop, divisorial
 contraction or Mori fibre space structure of a smooth projective
 toric variety $X$ is obtained by variation of the GIT parameter that is implicit in the
 Cox quotient construction. This is the case for toric quiver flag varieties, so every such rational
 map is of the form
 $\mathcal{M}_\vartheta=\Rep(Q,\underline{r})\git_\vartheta
 G\dashrightarrow \Rep(Q,\underline{r})\git_\theta G$ for some
 $\theta\in G^\vee$. Quiver flag varieties are not toric in general,
 but we can nevertheless say the following.

 \begin{proposition}
 \label{prop:MDS}
   Every quiver flag variety $\mathcal{M}_\vartheta$ is a Mori dream
   space.
 \end{proposition}
 \begin{proof}
   We show that the assumptions of Hu--Keel~\cite[Theorem~2.3]{HuKeel}
   are satisfied. The $\kk$-vector space $\Rep(Q,\underline{r})$ has
   trivial class group, and every $\vartheta$-semistable point is
   $\vartheta$-stable by Lemma~\ref{lem:generic}. Since
   $\mathcal{M}_\vartheta$ is projective, it remains to show that the
   $\vartheta$-unstable locus has codimension at least two.
   Proposition~\ref{prop:finemoduli}\two\ shows that each component of
   the $\vartheta$-unstable locus $\mathbb{V}(B)$ is cut out by a set
   of determinantal equations. A component of codimension one arises
   only if there exists a vertex $i$ with $r_i=s_i=1$, forcing the
   matrix $w_i$ in decomposition \eqref{eqn:coarsegrading} to be a
   $1\times 1$ matrix. In this case, $\mathcal{M}_\vartheta$ can be
   constructed equivalently by contracting the unique arrow $a$ with
   head at $i$ and identifying vertices $i$ and $\tail(a)$. Repeat
   until either no such vertices exist, or the quiver contracts to a
   single vertex, giving $\mathcal{M}_\vartheta=\Spec\kk$. This
   completes the proof.
 \end{proof}
 
 In the space of fractional characters 
 \[
 G^\vee_\QQ :=
 \Big\{(\theta_0,\dots,\theta_\rho)\in \QQ^{\rho+1} :
 \sum\theta_ir_i=0\Big\},
 \]
 the set of parameters $\theta\in G^\vee_\QQ$ for which
 $\Rep(Q,\underline{r})\git_\theta G$ is nonempty forms a closed
 polyhedral cone $\Sigma(Q,\underline{r})$. Two parameters are
 GIT-equivalent if their semistable loci coincide in
 $\Rep(Q,\underline{r})$, and the resulting equivalence classes
 provide a finite polyhedral decomposition of this cone into open GIT
 chambers separated by codimension-one walls. In our case,
 Hu--Keel~\cite[Lemma~2.2(4)]{HuKeel} implies that the map
 \begin{equation}
 \label{eqn:HuKeel}
 L\colon G^\vee_\QQ\to \Pic(\mathcal{M}_\vartheta)\otimes_\ZZ \QQ
 \end{equation}
 defined by sending $\theta = (\theta_0,\theta_1,\dots, \theta_\rho)$
 to $L(\theta)= \det(\mathscr{W}_1)^{\theta_1}\otimes \dots \otimes
 \det(\mathscr{W}_\rho)^{\theta_\rho}$ is an isomorphism that
 identifies the GIT chamber containing the special weight vector $\vartheta$ with the ample cone
 of $\mathcal{M}_\vartheta$. In particular,  $L(\vartheta)=\det(\mathscr{W}_1)\otimes \dots \otimes
 \det(\mathscr{W}_\rho)$ is the polarising ample line bundle inherited from the GIT description of $\mathcal{M}_\vartheta = \Rep(Q,\underline{r})\git_\vartheta G$.
 
 More generally, $L$ identifies the GIT
 chamber decomposition with a decomposition of the cone of pseudoeffective divisors $\overline{NE^1}(\mathcal{M}_\vartheta)$
 into chambers corresponding one-to-one with small $\QQ$-factorial
 modifications of $\mathcal{M}_\vartheta$.  Thus, all flips,
 flops, divisorial contractions or Mori fibre space structures of
 $\mathcal{M}_\vartheta$ are obtained by variation of GIT quotient, and all sequences of flips terminate.

 \begin{remark}
   Derksen--Weyman~\cite{DerksenWeyman2} describe the
   faces of the polyhedral cone $\Sigma(Q,\underline{r})$ for any
   acyclic quiver $Q$  and
   dimension vector $\underline{r}$ (they need not assume that $Q$ has a unique source). 
   \end{remark}
 
 We now describe the structure theorem for quiver flag varieties in terms of a tower of Mori fibre spaces. For $0\leq i\leq \rho$, let $Q(i)$ denote the subquiver of $Q$ obtained by removing all vertices $k>i$ and all arrows with head at
 any of these vertices. In other words, the vertex set is
 $Q(i)_0=\{0,\dots,i\}$ and the arrow set is $Q(i)_1 = \{a\in Q_1 :
 1\leq \head(a)\leq i\}$. For the quiver $Q(i)$ and for the dimension
 vector $\underline{r}(i) = (1,r_1,\dots, r_i)$ we let
 \[
 Y_i:=
 \mathcal{M}_{\vartheta(i)}\big(Q(i), \underline{r}(i)\big)
 \]
 denote the quiver flag variety, where $\vartheta(i) = (-\sum_{j=1}^i
 r_j,1, \dots, 1)\in \ZZ^{i+1}$ is the GIT parameter.  Write
 $\mathscr{W}^{(i)}_0,\mathscr{W}^{(i)}_1,\dots, \mathscr{W}^{(i)}_i$ for the
 tautological bundles on $Y_i$.

 \begin{theorem}
 \label{thm:tower}
 For any quiver flag variety $\mathcal{M}_\vartheta$ there is a tower 
 of Grassmann-bundles
 \begin{equation}
 \label{eqn:tower}
 \mathcal{M}_\vartheta=Y_\rho\longrightarrow Y_{\rho-1}\longrightarrow
 \dots \longrightarrow Y_1\longrightarrow Y_0=\Spec \kk.
 \end{equation}
 where for $1\leq i\leq \rho$, the variety $Y_i$ is isomorphic to the Grassmannian $\Gr(\mathscr{F}_i,
 r_i)$, where $\mathscr{F}_i$ is the locally free sheaf of rank $s_i$ on $Y_{i-1}$ given by 
 \[
 \mathscr{F}_i:=\bigoplus_{\{a\in
   Q_1 :\: \head(a)=i\}} \mathscr{W}^{(i-1)}_{\tail(a)}.
 \]
 \end{theorem}
 \begin{proof}
   Lemma~\ref{lem:generic} shows that for $1\leq i\leq \rho$, the
   projection map sending $(w_j)_{1\leq j\leq \rho}\in
   \Rep(Q,\underline{r})$ to $(w_j)_{1\leq j\leq i}\in
   \Rep\big(Q(i),\underline{r}(i)\big)$ restricts to a surjective map
   from the $\vartheta$-stable locus in $\Rep(Q,\underline{r})$ to the
   $\vartheta(i)$-stable locus in
   $\Rep\big(Q(i),\underline{r}(i)\big)$. This map is equivariant with
   respect to the change of basis group actions, giving surjective
   morphisms $\mathcal{M}_\vartheta\to Y_{i}$ for $i\geq 1$.
   Forgetting one vertex at a time shows that this map factors via
   $Y_j$ for $i<j<\rho$, and the structure morphism of $Y_1 = \Gr(\kk^{s_1},r_1)$ gives the final map in the
   tower \eqref{eqn:tower}. Now fix $i$ and consider the map $f\colon
   Y_i\to Y_{i-1}$ determined by the projection from $(w_j)_{1\leq
     j\leq i}$ that forgets the final matrix $w_i\in \Mat(r_i\times
   s_i,\kk)$.  Example~\ref{exa:Grassmannian} implies that the fibre
   of $f$ over any point $W = \big((W_j)_{0\leq j\leq i-1},(w_a)_{a\in
     Q(i-1)_1}\big)$ of $Y_{i-1}$ is the Grassmannian $\Gr(\kk^{s_i},r_i)$
   of $r_i$-dimensional quotients of the $s_i$-dimensional space
   $\bigoplus_{\{a\in Q_1 : \:\head(a)=i\}} W_{\tail(a)}$. As
   this point varies over $Y_{i-1}$, each $\kk$-vector space
   $W_{\tail(a)}$ sweeps out the corresponding tautological vector
   bundle $\mathscr{W}^{(i-1)}_{\tail(a)}$ on $Y_{i-1}$. This completes the
   proof.
 \end{proof}

\begin{remark}
The tower \eqref{eqn:tower} can be induced by variation of GIT, see the proof of Lemma~\ref{lem:nefcone}.
\end{remark}

 \begin{corollary}
 \label{coro:pullback}
 For $1\leq i\leq \rho$,  the tautological bundle $\mathscr{W}_i$ on $\mathcal{M}_\vartheta$ is
 obtained as the pullback of the tautological quotient bundle on
 $\Gr(\mathscr{F}_i, r_i)$. In addition:
 \begin{enumerate}
 \item[\one] the vector bundle $\mathscr{W}_i$ is globally generated for each $i\in Q_0$; and 
 \item[\two] the universal $\kk$-algebra homomorphism $\kk Q\to \End(\bigoplus_{i\in Q_0}\mathscr{W}_i)$ induces an isomorphism of $\kk$-vector spaces $e_0(\kk Q)e_i\cong H^0(\mathcal{M}_\vartheta, \mathscr{W}_i)$ for each $i\in Q_0$.
  \end{enumerate}
 \end{corollary}
 \begin{proof}
   As the $G$-orbit of the matrix $w_i\in \Mat(r_i\times s_i,\kk)$
   varies over points of $Y_i$, the surjective map
   $\bigoplus_{\{a\in Q_1 : \:\head(a)=i\}} W_{\tail(a)}\to W_i$
   sweeps out the tautological quotient $\mathscr{F}_i\to
   \mathscr{O}_{\Gr(\mathscr{F}_i,r_i)}(1)$. By pulling back to $\mathcal{M}_\vartheta$, the
   same surjective map of $\kk$-vector spaces defines a sheaf epimorphism
  \begin{equation}
  \label{eqn:epimorphism}
 \bigoplus_{\{a\in Q_1 :\: \head(a)=i\}}
   \mathscr{W}_{\tail(a)}\twoheadrightarrow \mathscr{W}_i.
 \end{equation}
 We now prove parts \one\ and \two\ simultaneously by induction. The result holds trivially for $i=0$ since $Q$ is acyclic and $\mathcal{M}_\vartheta$ is projective. Assume both statements for $j<i$. For $a\in Q_1$ satisfying $\head(a)=i$, there is a sheaf epimorphism $\mathscr{O}_{\mathcal{M}_\vartheta}^{\oplus d_{\tail(a)}}\twoheadrightarrow \mathscr{W}_{\tail(a)}$ for $d_{\tail(a)}=\dim_\kk e_0(\kk Q)e_{\tail(a)}$. By composing the sum of all such maps with \eqref{eqn:epimorphism}, we obtain a sheaf epimorphism 
 \[
 \mathscr{O}_{\mathcal{M}_\vartheta}^{\oplus d_i}=\bigoplus_{\{a\in Q_1 : \head(a)=i\}} \mathscr{O}_{\mathcal{M}_\vartheta}^{\oplus d_{\tail(a)}} \twoheadrightarrow \bigoplus_{\{a\in Q_1 : \head(a)=i\}} \mathscr{W}_{\tail(a)}\twoheadrightarrow \mathscr{W}_i,
 \]
 where $d_i:= \dim_\kk e_0(\kk Q)e_i$. This completes the proof.
    \end{proof}

 \begin{corollary}
 \label{coro:canonical}
 The anticanonical line bundle of $\mathcal{M}_\vartheta$ is 
 \[
  \omega_{\mathcal{M}_\vartheta}^{-1}  \cong   \bigotimes_{a\in Q_1} \Big(\det(\mathscr{W}_{\head(a)})^{r_{\tail(a)}}\otimes \det(\mathscr{W}_{\tail(a)})^{{-r_{\head(a)}}} \Big). 
  \]
 \end{corollary}
 \begin{proof}
 The given formula is equivalent to 
  \begin{equation}
 \label{eqn:canonical}
 \omega_{\mathcal{M}_\vartheta} = \bigotimes_{i\in Q_0} \det(\mathscr{W}_i)^{-s_i} \otimes \bigotimes_{a\in Q_1} \det(\mathscr{W}_{\tail(a)})^{{r_{\head(a)}}}.
 \end{equation}   
   We proceed by induction on $\rho = \vert Q_0\vert - 1$. For quivers
   with $\rho=1$, Example~\ref{exa:Grassmannian} shows that
   $\mathcal{M}_\vartheta\cong \Gr(\kk^{s_1},r_1)$ and that
   $\det(\mathscr{W}_1)$ is the ample bundle
   $\mathscr{O}_{\Gr(\kk^{s_1},r_1)}(1)$. The right-hand bundle from
   \eqref{eqn:canonical} is therefore $\det(\mathscr{W}_1)^{-s_1}\cong
   \mathscr{O}_{\Gr(\kk^{s_1},r_1)}(-s_1)$, the canonical bundle of
   the Grassmannian. We may assume by induction that the result holds
   for the quiver flag variety $Y_{\rho-1}$ whose defining quiver
   $Q(\rho-1)$ has one fewer vertices than $Q$. If we write $\pi\colon
   \mathcal{M}_\vartheta\to Y_{\rho-1}$ for the map from \eqref{eqn:tower} then as in Demailly~\cite[(2.10)]{Demailly} we have
   that
 \[
  \omega_{\mathcal{M}_\vartheta} \cong \pi^*\big(\omega_{Y_\rho-1}\otimes \det(\mathscr{F}_\rho)^{r_\rho}\big)\otimes \mathscr{O}_{\Gr(\mathscr{F}_\rho,r_\rho)}(-s_\rho).
 \]
 The inductive hypothesis and Corollary~\ref{coro:pullback} together give $\mathscr{O}_{\Gr(\mathscr{F}_\rho,r_\rho)}(-s_\rho)\cong
 \det(\mathscr{W}_\rho)^{-s_\rho}$ and
 \[
 \pi^*(\omega_{Y_{\rho-1}}) \cong \bigotimes_{i<\rho}
 \det(\mathscr{W}_i)^{-s_i}\otimes \bigotimes_{\{a\in Q_1 :\;
 \head(a)<\rho\}} \det(\mathscr{W}_{\tail(a)})^{r_{\head(a)}}.
 \]
 It remains to note that $\pi^*\big(\det(\mathscr{F}_\rho)^{r_\rho}\big)$ is isomorphic to $\bigotimes_{\{a \in Q_1 :\; \head(a) = \rho\}}
 \det(\mathscr{W}_{\tail(a)})^{r_\rho}$.
 \end{proof}
 
 The structure theorem, and in particular Corollary~\ref{coro:pullback}, shows that the tautological
 bundles generate a top-dimensional cone in the nef cone of $\mathcal{M}_\vartheta$.
 
 \begin{lemma}
   \label{lem:nefcone}
  The line bundles
   $\det(\mathscr{W}_1),\dots,\det(\mathscr{W}_\rho)$ are globally generated, and provide an integral basis
   for $\Pic(\mathcal{M}_\vartheta)$, so $\rho=\rank(\Pic(\mathcal{M}_\vartheta))$. Moreover, the ample cone of $\mathcal{M}_\vartheta$ contains 
  \begin{equation}
  \label{eqn:nefcone}
 \big\{\det(\mathscr{W}_1)^{\theta_1}\otimes \dots \otimes
 \det(\mathscr{W}_\rho)^{\theta_\rho} : \theta_1,\dots,\theta_\rho\in \QQ_{>0}\big\},
  \end{equation}
 and the bundles $\det(\mathscr{W}_1),\dots,\det(\mathscr{W}_{\rho-1})$ lie in a facet of the nef cone of $\mathcal{M}_\vartheta$.
   \end{lemma}
 \begin{proof}
   The bundle $\mathscr{W}_i$ is globally generated by
   Corollary~\ref{coro:pullback}, and one surjective map $\mathscr{O}^{\oplus m}\to
   \mathscr{W}_i$ produces another $\textstyle{\bigwedge^{r_i}}\mathscr{O}^{\oplus m}\rightarrow
   \det(\mathscr{W}_i)$, so $\det(\mathscr{W}_i)$ is globally generated.   Corollary~\ref{coro:pullback} also implies that $\det(\mathscr{W}_i)$ is the pullback to
   $\mathcal{M}_\vartheta$ of the ample bundle
   $\mathscr{O}_{\Gr(\mathscr{F}_i, r_i)}(1)$ that generates the
   relative Picard group $\Pic(Y_i/Y_{i-1})$, so $\det(\mathscr{W}_1),\dots,\det(\mathscr{W}_\rho)$ form an
   integral basis for $\Pic(\mathcal{M}_\vartheta)$ by Theorem~\ref{thm:tower}. The proof of Lemma~\ref{lem:generic} works equally well for any $\theta$ with $\theta_i>0$ for $i>0$, so the statement about the ample cone follows from the isomorphism $L$. For the final statement, note that the bundle
 $\bigotimes_{i=1}^{\rho-1} \det(\mathscr{W}_i)$ is of the form $L(\eta)$ for $\eta =(-\sum_{i=1}^{\rho-1} r_i,1, \dots,
 1,0)$. Variation of GIT quotient gives a morphism $g\colon
 \mathcal{M}_\vartheta \to \Rep(Q,\underline{r})\git_\eta G$. We claim that $\Rep(Q,\underline{r})\git_{\eta} G
 \cong Y_{\rho-1}$,  which implies that $g$ is a Mori fibre space and hence  $\bigotimes_{i=1}^{\rho-1} \det(\mathscr{W}_i)$ is not ample as required.  To prove the claim, set $Q^
\prime =Q(\rho-1)$ and $\underline{r}^\prime = \underline{r}(\rho-1)$  and let
 $G_\rho \cong \GL(r_\rho)$
 denote the subgroup of $G$ corresponding to change of basis at vertex $\rho\in Q_0$. The $G_\rho$-invariant
 subalgebra of $\kk[\Rep(Q,\underline{r})]$ is isomorphic to $\kk[\Rep(Q^\prime,\underline{r}^\prime)]$ and
 hence the affine quotient $\Rep(Q,\underline{r})/G_\rho$ is
 $\Rep(Q^\prime,\underline{r}^\prime)$. Taking
 the GIT quotient of $\Rep(Q,\underline{r})$ by $G$ linearised by
 $\eta$ is equivalent to first taking the affine quotient of
 $\Rep(Q,\underline{r})$ by $G_\rho$ and then taking the GIT quotient
 by the group $G/G_\rho\cong G(\underline{r}^\prime)$ linearised by $\vartheta^\prime = (-\sum_{i=1}^{\rho-1}r_i,1, \dots, 1)\in G(\underline{r}^\prime)^\vee$. This proves the claim.  \end{proof}

\begin{corollary}
The variety $\mathcal{M}_\vartheta$ is Fano if $s_i>s^\prime_i:=\sum_{\{a\in Q_1 : \tail(a) = i\}} r_{\head(a)}$. 
\end{corollary}
\begin{proof}
The formula from Corollary~\ref{coro:canonical} can be rewritten as
 $\omega_{\mathcal{M}_\vartheta}^{-1} \cong \bigotimes_{i\in Q_0} \det(\mathscr{W}_i)^{(s_i-s^\prime_i)}$.
\end{proof}

 \begin{examples}
 \label{exa:toric2}
 The following toric quiver flag varieties illustrate that while the line bundles $\det(\mathscr{W}_1),\dots,\det(\mathscr{W}_\rho)$ may generate the nef cone of  $\mathcal{M}_\vartheta$, this is not always the case.
 \begin{enumerate}
 \item[\one]  For the quiver $Q$ from Figure~\ref{fig:flag}(b), we have $\mathcal{M}_\vartheta \cong
 \mathbb{P}_{\mathbb{P}^1}\big(\mathscr{O}\oplus\mathscr{O}(1)\oplus\mathscr{O}(1)\big)$
 and
 $\omega_{\mathcal{M}_\vartheta}\cong \mathscr{W}_2^{-3}$ by Theorem~\ref{thm:tower} and Corollary~\ref{coro:canonical} respectively.  The 
 bundles $\mathscr{W}_1$ and $\mathscr{W}_2$ generate the nef cone, so
 $\mathcal{M}_\vartheta$ is not Fano. There are two GIT chambers: varying the GIT parameter beyond the nef cone from
 $\vartheta$ to $\theta:=(-1,-1,2)$ induces a flop of
 $\mathcal{M}_\vartheta$.
 \item[\two] The quiver $Q$ with three vertices and four arrows obtained from that in Figure~\ref{fig:flag}(b) by removing arrow 3 defines $\mathcal{M}_\vartheta \cong
   \mathbb{P}_{\mathbb{P}^1}\big(\mathscr{O}(1)\oplus\mathscr{O}(1)\big)\cong
   \mathbb{P}^1\times \mathbb{P}^1$. The map
   $\varphi_{\vert \mathscr{W}_2\vert}\colon
   \mathcal{M}_\vartheta\to \Rep(Q,\underline{r})\git_{(-1,0,1)} G\cong
   \mathbb{P}^3$ embeds $\mathbb{P}^1\times \mathbb{P}^1$ as a quadric hypersurface, so $\mathscr{W}_2$ is very ample.
\end{enumerate}
\end{examples}
   
 \section{A tilting bundle from tautological bundles}
 \label{sec:four}
 This section investigates the bounded derived category of coherent sheaves on a quiver flag variety. By applying the construction of Kapranov~\cite{Kapranov1} inductively to the tower of Grassmann bundles from Theorem~\ref{thm:tower}, we exhibit a tilting bundle on $\mathcal{M}_\vartheta$ and hence describe its bounded derived category of coherent sheaves.  We assume throughout that $\mathcal{M}_\vartheta$ is defined by a quiver $Q$ and dimension vector $\underline{r}$ satisfying strictly the numerical condition \eqref{eqn:numerical}.
 
 We first recall the Schur powers of a locally free sheaf. A partition $\lambda = (\lambda_1,\dots,\lambda_r)$ satisfying $\lambda_1\geq \dots \geq \lambda_r\geq 0$ is represented by its Young diagram comprising $\sum_i \lambda_i$ boxes arranged in left-justified rows, where the $i$th row comprises $\lambda_i$ boxes. Let $\Young(k,r)$ denote the set of Young diagrams with no more than $k$ columns and no more than $r$ rows, so $k\geq \lambda_1\geq \dots \geq \lambda_r\geq 0$. Given a vector space $E$ of dimension $r$ and a partition $\lambda$ with no more than $r$ rows, its Schur power $\mathbb{S}^{\lambda}E$  is the irreducible $\GL(E)$-module with highest weight $\lambda$. Familiar examples include $\mathbb{S}^{(d,0,\dots,0)}E = \Sym^d(E)$ and $\mathbb{S}^{(1,\dots,1,0,\dots,0)}E = \bigwedge^d(E)$ where the latter partition has $d$ 1's. In particular,  for $d=r$ we have $\det(E)=\mathbb{S}^{(1,\dots,1)}E$, and the construction extends to sequences $\lambda$ involving negative terms by defining 
 \[
 \mathbb{S}^\lambda E:= \mathbb{S}^{(\lambda_1+m,\dots,\lambda_d+m)} E\otimes \det(E)^{-m}
 \]
 for $m\in \NN$. Given partitions $\lambda, \mu$, the Littlewood--Richardson rules govern the decomposition of the tensor product of two Schur powers and the Schur power of a direct sum, namely
  \begin{equation}
 \label{eqn:LRrules}
 \mathbb{S}^{\lambda}E \otimes \mathbb{S}^{\mu}E = \bigoplus_{\nu} \big(\mathbb{S}^\nu E\big)^{\oplus c_{\lambda,\mu}^\nu}\quad\text{and}\quad\mathbb{S}^\nu(E\oplus F)  = \bigoplus_{\lambda, \mu} (\mathbb{S}^\lambda E\oplus \mathbb{S}^\mu F)^{\oplus c_{\lambda, \mu}^{\nu}}
 \end{equation}
 where $c_{\lambda,\mu}^\nu$ are the Littlewood--Richardson coefficients (see Fulton~\cite[\S8.3]{Fulton}). More generally, the construction of Schur powers and the decomposition formulae \eqref{eqn:LRrules} are valid for any locally free sheaf $\mathscr{E}$ of rank $r$, giving rise to locally free sheaves $\mathbb{S}^{\lambda}\mathscr{E}$.  
   
\begin{lemma}
\label{lem:Rpushforward}
Let $\mathcal{M}_\vartheta$ be a quiver flag variety. For any partition $\lambda$,  the higher direct images of $\mathbb{S}^\lambda\mathscr{W}_\rho$ under the Grassmann-bundle $\pi\colon \mathcal{M}_\vartheta=\Gr(\mathscr{F}_\rho,r_\rho)\to Y_{\rho-1}$ from \eqref{eqn:tower} satisfy 
 \[
 R^k\pi_*(\mathbb{S}^\lambda\mathscr{W}_\rho)=\left\{\begin{array}{cl} \mathbb{S}^\lambda\mathscr{F}_\rho & \text{if }k=0 \text{ and if }\lambda_1\geq \dots \geq \lambda_{r_\rho}\geq 0\\ 0 & \text{otherwise.}\end{array}\right.
 \]
 \end{lemma}
 \begin{proof} 
 For $y\in Y_{\rho-1}$,  we have $\pi^{-1}(y)\cong\Gr(F_\rho,r_{\rho})$ where $F_\rho:=(F_\rho)_y$ is the fibre of $\mathscr{F}_\rho$ over $y$.  By cohomology and base change, the fibre of $R^k\pi_*(\mathbb{S}^\lambda\mathscr{W}_\rho)$ over $y$ is $H^k\big(\!\Gr(F_{\rho},r_{\rho}),\mathbb{S}^\lambda\mathscr{W}_\rho\vert_{\Gr(F_{\rho},r_{\rho})}\big)$. Corollary~\ref{coro:pullback} shows that $\mathscr{W}_\rho$ is the rank $r_\rho$ quotient bundle on $\Gr(\mathscr{F}_\rho,r_\rho)$, so the restriction of $\mathscr{W}_\rho$  to $\Gr(F_{\rho},r_{\rho})$ is the rank $r_\rho$ quotient bundle $\mathscr{E}$ on $\Gr(F_{\rho},r_{\rho})$ and hence the restriction of $\mathbb{S}^\lambda\mathscr{W}_\rho$ to $\Gr(F_{\rho},r_{\rho})$ is $\mathbb{S}^\lambda \mathscr{E}$. Kapranov~\cite[Proposition~2.2(a)]{Kapranov1} implies that   
 \[
 H^k\big(\!\Gr(F_{\rho},r_{\rho}),\mathbb{S}^\lambda\mathscr{E}\big)=\left\{\begin{array}{cl} \mathbb{S}^\lambda F_\rho & \text{if }k=0 \text{ and if }\lambda_1\geq \dots \geq \lambda_{r_\rho}\geq 0\\ 0 & \text{otherwise.}\end{array}\right.
 \]
The result follows by varying the point $y\in Y_{\rho-1}$.
 \end{proof}

\begin{proposition}
\label{prop:quivervanishing}
Let $\lambda_1,\dots,\lambda_\rho$ be partitions. Then 
\[
H^k\big(\mathcal{M}_\vartheta,\mathbb{S}^{\lambda_1}\mathscr{W}_1\otimes \dots \otimes \mathbb{S}^{\lambda_\rho}\mathscr{W}_\rho\big) = 0 \quad\text{for } k>0
\]
if each partition $\lambda_i=(\lambda_{i,1},\dots,\lambda_{i,r_{i}})$ satisfies $\lambda_{i,j}\geq -(s_i-r_i)$ for all $1\leq j\leq r_i$. 
\end{proposition}
\begin{proof}
We proceed by induction using the tower of Grassmann bundles from Theorem~\ref{thm:tower}. If $\rho=1$ then the isomorphism $\mathcal{M}_\vartheta\cong \Gr(\kk^{s_1},r_1)$ identifies $\mathscr{W}_1$ with the quotient bundle  of rank $r_1$ on $\Gr(\kk^{s_1},r_1)$, in which case the result is due to Kapranov~\cite[Proposition~2.2(a)]{Kapranov1}.  Assume by induction that the result holds for the quiver flag variety $Y_{\rho-1}$. Simplify notation by relabeling the tautological bundles on $Y_{\rho-1}$ as $\mathscr{V}_i:= \mathscr{W}^{(\rho-1)}_i$ for $i\leq \rho-1$.  Corollary~\ref{coro:pullback} implies that the pullback of $\mathscr{V}_i$ via $\pi\colon \mathcal{M}_\vartheta=\Gr(\mathscr{F}_\rho,r_\rho) \to Y_{\rho-1}$ is isomorphic to $\mathscr{W}_i$ for $i\leq \rho-1$, so 
 \[
 \mathbb{S}^{\lambda_1}\mathscr{W}_1\otimes \dots \otimes \mathbb{S}^{\lambda_\rho}\mathscr{W}_\rho\cong \pi^*\big( \mathbb{S}^{\lambda_1}\mathscr{V}_1\otimes \dots \otimes \mathbb{S}^{\lambda_{\rho-1}}\mathscr{V}_{\rho-1} \big)\otimes \mathbb{S}^{\lambda_{\rho}}\mathscr{W}_{\rho}.
 \]
 The projection formula gives
\[
\Rderived\pi_*\big(\mathbb{S}^{\lambda_1}\mathscr{W}_1\otimes \dots \otimes \mathbb{S}^{\lambda_\rho}\mathscr{W}_\rho\big) \cong \mathbb{S}^{\lambda_1}\mathscr{V}_1\otimes \dots \otimes \mathbb{S}^{\lambda_{\rho-1}}\mathscr{V}_{\rho-1}\otimes \Rderived\pi_*\big(\mathbb{S}^{\lambda_\rho}\mathscr{W}_\rho\big).
\]
 Lemma~\ref{lem:Rpushforward} implies that this complex is nonzero only when $\lambda_\rho = (\lambda_{\rho,1},\dots, \lambda_{\rho,r_{\rho}})$ satisfies $\lambda_{\rho,j}\geq 0$ for $1\leq j\leq r_\rho$, in which case it is equal to $ \mathbb{S}^{\lambda_1}\mathscr{V}_1\otimes \dots \otimes \mathbb{S}^{\lambda_{\rho-1}}\mathscr{V}_{\rho-1} \otimes \mathbb{S}^{\lambda_\rho}\mathscr{F}_\rho$ as a complex concentrated in degree zero.  Since $\mathscr{F}_\rho = \bigoplus_{\{a\in Q_1 : \head(a)=\rho\}} \mathscr{W}_{\tail(a)}$, the second Littlewood--Richardson rule from \eqref{eqn:LRrules} show that $\mathbb{S}^{\lambda_\rho}\mathscr{F}_\rho$ decomposes as a direct sum of locally free sheaves of the form $\mathbb{S}^{\mu_1}\mathscr{V}_1\otimes \dots \otimes \mathbb{S}^{\mu_{\rho-1}}\mathscr{V}_{\rho-1}$, where each partition $\mu_i=(\mu_{i,1},\dots,\mu_{i,r_i})$ satisfies $\mu_{i,j}\geq 0$. The first rule from \eqref{eqn:LRrules} then implies that $\mathbb{S}^{\lambda_1}\mathscr{V}_1\otimes \dots \otimes \mathbb{S}^{\lambda_{\rho-1}}\mathscr{V}_{\rho-1} \otimes \mathbb{S}^\lambda\mathscr{F}_\rho$ decomposes as a direct sum of bundles of the form $\mathbb{S}^{\nu_1}\mathscr{V}_1\otimes \dots \otimes \mathbb{S}^{\nu_{\rho-1}}\mathscr{V}_{\rho-1}$, where each $\nu_i=(\nu_{i,1},\dots,\nu_{i,r_i})$ satisfies $\nu_{i,j}\geq -(s_i-r_i)$. The higher cohomology groups of such bundles vanish by the inductive hypothesis, giving 
\[
H^p\big(Y_{\rho-1},R^q\pi_*(\mathbb{S}^{\lambda_1}\mathscr{W}_1\otimes \dots \otimes \mathbb{S}^{\lambda_\rho}\mathscr{W}_\rho)\big)=0
\]
for $p,q\geq 0$ with $(p,q)\neq (0,0)$, as long as each $\lambda_i$ satisfies $\lambda_{i,j}\geq -(s_i-r_i)$ for $1\leq j\leq r_{i}$. The degeneration of the Leray spectral sequence then implies that 
 \[
 H^k\big(\mathcal{M}_\vartheta,\mathbb{S}^{\lambda_1}\mathscr{W}_1\otimes \dots \otimes \mathbb{S}^{\lambda_\rho}\mathscr{W}_\rho\big)
\cong H^p\big(Y_{\rho-1},R^q\pi_*(\mathbb{S}^{\lambda_1}\mathscr{W}_1\otimes \dots \otimes \mathbb{S}^{\lambda_\rho}\mathscr{W}_\rho)\big)=0
 \]
for $k=p+q > 0$ and for the stated conditions on the partitions $\lambda_1,\dots,\lambda_\rho$. This proves the proposition.
\end{proof}

\begin{remark} 
\label{rem:toricvanishing}
If $\mathcal{M}_\vartheta$ is a  toric quiver flag variety, Proposition~\ref{prop:quivervanishing} states simply that
 \begin{equation}
 \label{eqn:toricvanishing}
 H^k\big(\mathcal{M}_\vartheta,\mathscr{W}_1^{\theta_1}\otimes \dots \otimes \mathscr{W}_\rho^{\theta_\rho}\big) = 0
  \end{equation}
 whenever $\theta_i > -s_i$ for all $1\leq i\leq \rho$.  Since each $\mathscr{W}_i$ is nef, the higher cohomology groups are known to vanish if $\theta_i\geq 0$ for all $1\leq i\leq \rho$, so \eqref{eqn:toricvanishing} extends this to include a limited range of negative tensor powers of the tautological bundles. By computing the cones of nonvanishing higher cohomology from Eisenbud--Musta{\c{t}}{\v{a}}--Stillman~\cite{EMS}, one can show that this is optimal in the sense that for each $1\leq i\leq \rho$, there exists $(\theta_1,\dots,\theta_\rho)$ with $\theta_j>-s_j$ for $j\neq i$ and $\theta_i=-s_i$ such that $H^{s_i-1}\big(\mathcal{M}_\vartheta,\mathscr{W}_1^{\theta_1}\otimes \dots \otimes \mathscr{W}_\rho^{\theta_\rho}\big) \neq 0$. To illustrate this we present an example.
 \end{remark} 

\begin{example}
\label{exa:nonvanishingcones}
For the quiver $Q$ shown in Figure~\ref{fig:cohomologycones}, Theorem~\ref{thm:tower} shows that the corresponding toric quiver flag variety is $\mathcal{M}_\vartheta\cong \mathbb{P}_{Y_2}\big(\mathscr{O}(1,0)\oplus\mathscr{O}(0,1)^{\oplus 2}\big)$, where $Y_2\cong \mathbb{P}_{\mathbb{P}^1}\big(\mathscr{O}\oplus \mathscr{O}(1)^{\oplus 4}\big)$. If we identify $\Pic(\mathcal{M}_\vartheta)$ with $\ZZ^3$ via the isomorphism sending $\mathscr{W}_1^{\theta_1}\otimes \mathscr{W}_2^{\theta_2}\otimes \mathscr{W}_3^{\theta_3}$ to $(\theta_1,\theta_2,\theta_3)$, then the ample cone of $\mathcal{M}_\vartheta$ is the positive octant and $\omega_{\mathcal{M}_\vartheta}^{-1} = (-3,3,3)$, so $\mathcal{M}_\vartheta$ is not Fano. 
  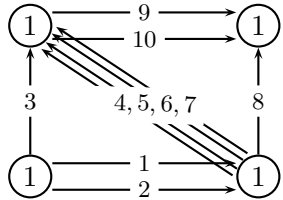
\begin{figure}[!ht]
    \centering
        \psset{unit=1cm}
        \begin{pspicture}(0,0)(3.5,2.2)
          \cnodeput(0.5,0){A}{1}
          \cnodeput(3.5,0){B}{1} 
          \cnodeput(0.5,2){C}{1}
          \cnodeput(3.5,2){D}{1}
          \psset{nodesep=0pt}
          \ncline[offset=5pt]{->}{A}{B}\lput*{:U}{\footnotesize{1}}
          \ncline[offset=-5pt]{->}{A}{B}\lput*{:U}{\footnotesize{2}}
          \ncline{->}{A}{C} \lput*{:270}{\footnotesize{3}}
          \ncline[offset=5pt]{<-}{C}{B}
          \ncline[offset=5pt]{->}{B}{C}
         \ncline[offset=2pt]{<-}{C}{B}
          \ncline[offset=2pt]{->}{B}{C}\lput*{:212}(0.5){\footnotesize{$\;\; \;\; 4,5,6,7$}}
           \ncline{->}{B}{D} \lput*{:270}{\footnotesize{8}}
           \ncline[offset=5pt]{->}{C}{D}\lput*{:U}{\footnotesize{9}}
          \ncline[offset=-5pt]{->}{C}{D}\lput*{:U}{\footnotesize{10}}
        \end{pspicture}
    \caption{The quiver $Q$ defining a $\mathbb{P}^2$-bundle over $\mathbb{P}_{\mathbb{P}^1}\big(\mathscr{O}\oplus \mathscr{O}(1)^{\oplus 4}\big)$. \label{fig:cohomologycones}}
  \end{figure}
  The cones in $\Pic(\mathcal{M}_\vartheta)\otimes_{\ZZ}\QQ$ comprising fractional line bundles with nonvanishing higher cohomology can be computed using Macaulay2~\cite{M2}. Table~\ref{tab:nonvanishingcones} shows that the cones correspond one-to-one with subsets $I$ of $\{1,2,3\}$, where $d_a:= \mathscr{W}_{\head(a)}\otimes\mathscr{W}_{\tail(a)}^{-1}$ is the degree of $y_a\in \kk[y_a : a\in Q_1]$. The region $\{(\theta_1,\theta_2,\theta_3) : \theta_1>-2,\theta_2>-5,\theta_3>-3\}$ avoids each cone and, moreover,  $H^1(\mathscr{W}_1^{-2})\neq 0$, $H^4(\mathscr{W}_1^{4}\otimes \mathscr{W}_2^{-5})\neq 0$ and $H^2(\mathscr{W}_1\otimes \mathscr{W}_2^{2}\otimes \mathscr{W}_3^{-3})\neq 0$, so the bounds from \eqref{eqn:toricvanishing} are sharp.
\begin{table}[!ht]
\begin{center}
\begin{tabular}{c|c|c|c}
  & $I$ & $\sum_{i\in I}(s_i-1)$ & $ \text{The cone corresponding to }I$ \\ \hline
 $H^1\neq 0$ & $\{1\}$ & 1 & $(-2,0,0) + \pos(-d_1,-d_2,d_3,d_4,\dots,d_{10})$  \\
 $H^2\neq 0$ & $\{3\}$ & 2 & $(1,2,-3) + \pos(d_1,\dots,d_7,-d_8,-d_9,-d_{10})$  \\
$H^3\neq 0$ & $\{1,3\}$ & 3 & $(-1,2,-3) + \pos(-d_1,-d_2,d_3,\dots,d_7,-d_8,-d_9,-d_{10})$  \\
$H^4\neq 0$ & $\{2\}$ & 4 & $(4,-5,0) + \pos(d_1,d_2,-d_3,\dots,-d_7,d_8,d_9,d_{10})$  \\
$H^5\neq 0$ & $\{1,2\}$ & 5 & $(2,-5,0) + \pos(-d_1,\dots,-d_7,d_8,d_9,d_{10})$  \\
$H^6\neq 0$ & $\{2,3\}$ & 6 & $(5,-3,-3) + \pos(d_1,d_2,-d_3,\dots,-d_{10})$  \\
$H^7\neq 0$ & $\{1,2,3\}$ & 7 & $(3,-3,-3) + \pos(-d_1,\dots,-d_{10})$  
\end{tabular}
\end{center}
\caption{Cones of nonvanishing higher cohomology}\label{tab:nonvanishingcones}
\end{table}
\end{example}

To state the main result of this section, we recall some standard notions from the study of derived categories.  A sheaf $\mathscr{E}$ on a scheme $X$ is \emph{exceptional} if $\Hom(\mathscr{E},\mathscr{E})=\kk$ and $\Ext^k(\mathscr{E},\mathscr{E})=0$ for $k>0$. A sequence of exceptional sheaves $(\mathscr{E}_1,\dots,\mathscr{E}_m)$ on $X$ is \emph{exceptional} if $\Ext^k(\mathscr{E}_j,\mathscr{E}_i)=0$ for $k\in \ZZ$ with $j>i$, and it is \emph{strongly exceptional} if in addition $\Ext^k(\mathscr{E}_i,\mathscr{E}_j) = 0$ for $k>0$ with $i<j$.  A sequence $(\mathscr{E}_1,\dots,\mathscr{E}_m)$ of coherent sheaves on $X$ is \emph{full} if the smallest triangulated subcategory of the bounded derived category of coherent sheaves $D^b(\coh(X))$ that contains all direct summands of each sheaf $\mathscr{E}_1,\dots,\mathscr{E}_m$ is equal to the whole of $D^b(\coh(X))$.  

The following result generalises that of Kapranov~\cite{Kapranov2} for a flag variety, though here we work with tautological quotient bundles rather than sub-bundles.
  
 \begin{theorem}
 \label{thm:quivertilt}
 For any quiver flag variety $\mathcal{M}_\vartheta$, the locally free sheaves on $\mathcal{M}_\vartheta$ of the form
  \begin{equation}
  \label{eqn:quivertilt}
 \Big\{\mathbb{S}^{\lambda_1}\mathscr{W}_1\otimes \dots \otimes \mathbb{S}^{\lambda_\rho}\mathscr{W}_\rho : \lambda_i \in \Young(s_i-r_i,r_i) \text{ for } 1\leq i\leq \rho\Big\}
 \end{equation}
 admit an order so that the resulting sequence is full and strongly exceptional. 
 \end{theorem}
 \begin{proof} 
We proceed by induction using the tower from Theorem~\ref{thm:tower}, adopting the same notation as in the proof of Proposition~\ref{prop:quivervanishing}. If $\rho=1$ then $\mathcal{M}_\vartheta\cong \Gr(\kk^{s_1},r_1)$, in which case the result is due to Kapranov~\cite{Kapranov1}.  Assume by induction that the result holds for the quiver flag variety $Y_{\rho-1}$. The set of locally free sheaves from \eqref{eqn:quivertilt} coincides with the set 
  \[
  \Big\{\pi^*\big( \mathbb{S}^{\lambda_1}\mathscr{V}_1\otimes \dots \otimes \mathbb{S}^{\lambda_{\rho-1}}\mathscr{V}_{\rho-1} \big)\otimes \mathbb{S}^{\lambda_{\rho}}\mathscr{W}_{\rho} \; :\; \lambda_i \in \Young(s_i-r_i,r_i) \text{ for } 1\leq i\leq \rho\Big\}, 
 \]
 The inductive hypothesis shows that the bundles of the form $\mathbb{S}^{\lambda_1}\mathscr{V}_1\otimes \dots \otimes \mathbb{S}^{\lambda_{\rho-1}}\mathscr{V}_{\rho-1} $ provide a full exceptional sequence on $Y_{\rho-1}$. Since $\mathscr{W}_\rho$ is isomorphic to the rank $r_\rho$ quotient bundle on $\Gr(\mathscr{F}_\rho,r_\rho)$, applying Kapranov's result to the Grassmann-bundle $\Gr(\mathscr{F}_\rho,r_\rho)$ following Orlov~\cite[\S3]{Orlov} implies that the bundles \eqref{eqn:quivertilt} can be ordered to form 
a full exceptional sequence on $\mathcal{M}_\vartheta$. It remains to prove that the higher $\Ext$ groups vanish between any pair of bundles from \eqref{eqn:quivertilt} or, equivalently, that
 \[
 H^k\Big(\mathcal{M}_\vartheta, \bigotimes_{1\leq i\leq \rho}\mathscr{H}\emph{om}\big(\mathbb{S}^{\nu_i}\mathscr{W}_i,  \mathbb{S}^{\mu_i}\mathscr{W}_i\big)\Big) = 0\quad \text{for }k>0
\] 
 whenever $\nu_i,\mu_i\in \Young(s_i-r_i,r_i)$ for $1\leq i\leq \rho$.  We have $\mathscr{H}\emph{om}\big(\mathbb{S}^{\nu_i}\mathscr{W}_i,  \mathbb{S}^{\mu_i}\mathscr{W}_i\big) = \mathbb{S}^{\mu_i-\nu_i}\mathscr{W}_i$ where the subtraction is performed componentwise. For $1\leq i\leq \rho$, the difference $\lambda_i:=\mu_i-\nu_i$ is a sequence $(\lambda_{i,1},\dots,\lambda_{i,r_i})$ satisfying $-(s_i-r_i)\leq \lambda_{i,j}\leq s_i-r_i$ for $1\leq j\leq r_i$, so the necessary vanishing follows from Proposition~\ref{prop:quivervanishing}.
\end{proof}

\begin{remark}
 For the quiver $Q$ with vertex set $Q_0=\{0,1,2\}$ and arrow set comprising $m$ arrows from 0 to 1 and $n$ arrows from 1 to 2, and for any dimension vector $\underline{r}$ satisfying \eqref{eqn:numerical}, the result of Theorem~\ref{thm:quivertilt} is due to Halic~\cite[Example~8.2.1]{Halic}. In this case $\mathcal{M}_\vartheta(Q,\underline{r})$ is Fano; in fact \cite{Halic} uses the anticanonical character rather than $\vartheta$, but both lie in the cone \eqref{eqn:nefcone}. 
\end{remark}

For any quiver flag variety $\mathcal{M}_\vartheta$, define the locally free sheaf 
\[
 \mathscr{T}:=\bigoplus_{1\leq i\leq \rho}\bigoplus_{\lambda_i\in \Young(s_i-r_i,r_i)} \mathbb{S}^{\lambda_1}\mathscr{W}_1\otimes \dots \otimes \mathbb{S}^{\lambda_\rho}\mathscr{W}_\rho
 \]
 on $\mathcal{M}_\vartheta$. Set $A=\End(\mathscr{T})$, and let $\modA$ denote the abelian category of finitely generated right $A$-modules. An observation of Baer~\cite{Baer} and Bondal~\cite{Bondal} provides the following simple description of the bounded derived category of coherent sheaves on $\mathcal{M}_\vartheta$.
  
 \begin{corollary}
 For any quiver flag variety $\mathcal{M}_\vartheta$, the functor 
 \[
 \Rderived\Hom(\mathscr{T},-)\colon
 D^b\big(\!\coh(\mathcal{M}_\vartheta)\big)\longrightarrow D^b\big(\!\modA\big)
 \]
 is an exact equivalence of triangulated categories.
 \end{corollary}
 \begin{proof}
 It is well known that the direct sum of all sheaves in a full strong, exceptional collection is a tilting sheaf; see, for example, King~\cite{King2}.
 \end{proof}
 
 \begin{remark} For a  toric quiver flag variety, the set \eqref{eqn:quivertilt} is simply
 \begin{equation}
  \label{eqn:torictilt}
 \Big\{\mathscr{W}_1^{\theta_1}\otimes \dots \otimes \mathscr{W}_\rho^{\theta_\rho} : 0\leq \theta_i < s_i \text{ for } 1\leq i\leq \rho\Big\},
 \end{equation}
 so every toric quiver flag variety admits a tilting bundle whose summands are line bundles.
   \begin{enumerate}
   \item[\one] Existence of a tilting bundle on  $\mathcal{M}_\vartheta$ which decomposes into a direct sum of line bundles follows from Costa--Miro-Roig~\cite[Proposition~4.9]
{CostaMiroRoig} since $\mathcal{M}_\vartheta$ has a splitting fan. Here we are able to exhibit an explicit tilting bundle since we need not appeal to Serre vanishing. 
However, our hypotheses on the toric variety are more restrictive since, for example,  $\mathbb{F}_2$ has a splitting fan but it is not a toric quiver flag variety.  
    \item[\two] The set \eqref{eqn:torictilt} can be characterised as the line bundles obtained as nonnegative powers of the tautological 
bundles on $\mathcal{M}_\vartheta$ that arise in the Buchsbaum-Rim complex of the map
    \[
    \bigoplus_{a\in Q_1} \mathscr{W}_{\tail(a)}\boxtimes \mathscr{W}_{\head(a)}^{-1} \longrightarrow   \bigoplus_{i\in Q_0} \mathscr{W}_{i}\boxtimes \mathscr{W}_{i}^{-1}
    \]
   given by $w_a\boxtimes 1 - 1\boxtimes w_a^*$, where $w_a\colon \mathscr{W}_{\tail(a)}\to \mathscr{W}_{\head(a)}$ are the canonical maps for $a
  \in Q_1$. This observation follows from a calculation by Broomhead~\cite[Section~4.1]{Broomhead}.
    \item[\three] In the special case where $\mathcal{M}_\vartheta$  is Fano, 
Altmann--Hille~\cite{AltmannHille} implies that the bundles $\mathscr{W}_0,\dots, \mathscr{W}_\rho$ generate a full and faithful triangulated subcategory of $D^b\big(\!
\coh(\mathcal{M}_\vartheta)\big)$. Our Theorem~\ref{thm:quivertilt} extends this sequence to provide all summands of a tilting generator for  $D^b\big(\!
\coh(\mathcal{M}_\vartheta)\big)$. Note that \cite{AltmannHille} need not assume that $Q$ has a 
unique source (though it does require that the canonical linearisation is generic).
     \end{enumerate}
    \end{remark}

 \section{Multigraded linear series}
 \label{sec:five}
 We now demonstrate how quiver flag varieties arise naturally as
 ambient spaces in algebraic geometry. This description generalises the
 classical construction of the linear series associated to a line
 bundle on an algebraic variety or, more generally, the Grassmannian
 associated to a vector bundle of higher rank. 
  
 Let $X$ be a projective scheme defined over an algebraically closed
 field $\kk$. For $\rho\geq 1$, consider a sequence $(\mathscr{E}_1,\dots, \mathscr{E}_\rho)$ of distinct,
 nontrivial locally free sheaves on $X$, each with $H^0(X,\mathscr{E}_i)\neq 0$. Augment this with the
 trivial bundle $\mathscr{E}_0:=\mathscr{O}_X$ to obtain the sequence 
  \[
  \underline{\mathscr{E}}:=\big(
 \mathscr{O}_X, \mathscr{E}_1, \dots, \mathscr{E}_\rho\big).
  \]
 We say that the sequence $\underline{\mathscr{E}}$ is \emph{weakly exceptional} if
 $\Hom(\mathscr{E}_j,\mathscr{E}_i)=0$ for $j>i$.    
 
 Every weakly exceptional sequence $\underline{\mathscr{E}}$ determines a finite, acyclic quiver $Q$ with a
 unique source, called the \emph{quiver of sections of $\underline{\mathscr{E}}$}. The vertex set $Q_0=\{0,1,\dots, \rho\}$ corresponds
 to the bundles in $\underline{\mathscr{E}}$. As for the arrow set $Q_1$, for each
 $i,j\in Q_0$ let $V_{i,j}\subseteq \Hom(\mathscr{E}_{i},\mathscr{E}_j)$ denote the
 $\kk$-vector subspace spanned by maps factoring via $\mathscr{E}_k$ for some
 $k\neq i,j$.  The number of arrows in $Q$ from $i$ to $j$ is defined
 to be $n_{i,j}:=\dim \big(\Hom(\mathscr{E}_{i},\mathscr{E}_j)/V_{i,j}\big)$ if $i\neq j$,
 and is set to be zero if $i=j$. Projectivity of $X$ ensures that $Q_1$ is finite, while $Q$ is acyclic and has $0\in Q_0$ as the unique source since the sequence is weakly exceptional. 
  
 \begin{example}
 \label{exa:flagqos}
   For $\rho > 0$, fix integers $n>r_1>r_2>\dots > r_\rho>0$ and write
   $X:=\Fl(n;r_1,\dots,r_\rho)$ for the variety of quotient flags (see
   Example~\ref{exa:flag}). For $1\leq j\leq \rho$, let
   $\mathscr{O}_X^{\oplus n}\to \mathscr{E}_j$ denote the map of
   vector bundles whose fibre over a point $x\in X$ is the surjective map
   $\kk^n\to E_j$ obtained from the corresponding flag. To compute the quiver of sections of  $\underline{\mathscr{E}}$, note that each $\mathscr{E}_j$ is globally generated, and every element in $\Hom(\mathscr{O}_X,\mathscr{E}_j)$ factors via $\mathscr{E}_i$ for all $i<j$ since so does the corresponding map of the flag. Thus far, then, the quiver of sections coincides with the quiver $Q$ shown in Figure~\ref{fig:flag}(a). Kapranov~\cite{Kapranov2} shows that the sequence $\underline{\mathscr{E}}:=(
 \mathscr{O}_X, \mathscr{E}_1, \dots, \mathscr{E}_\rho)$ is strongly exceptional, and hence it is weakly exceptional. It follows that the quiver of sections has no additional arrows and hence is equal to $Q$. 
 \end{example}

 \begin{remark}
 \label{rem:framings2}
 The process of augmenting the sequence $(\mathscr{E}_1,\dots, \mathscr{E}_\rho)$
 with the trivial bundle to produce $\underline{\mathscr{E}}$ is the
 geometric analogue of the framing of a quiver by the auxiliary
 dimension vector $(n_{0,1}, \dots, n_{0,\rho})$ to produce the quiver of sections $Q$; compare
 Remark~\ref{rem:Ginzburg}.
 \end{remark}

The \emph{multigraded linear series} of a weakly exceptional sequence $\underline{\mathscr{E}}=\big( \mathscr{O}_X, \mathscr{E}_1,
 \dots, \mathscr{E}_\rho\big)$ with quiver of sections $Q$ is the quiver flag variety $\vert\underline{\mathscr{E}}\vert :=
 \mathcal{M}_\vartheta(Q,\underline{r})$ for $\underline{r} = (1,\rank(\mathscr{E}_1),\dots,
 \rank(\mathscr{E}_\rho))$.  A multigraded linear series may a priori be empty, but when it is not it carries tautological vector bundles
 $\mathscr{W}_0, \mathscr{W}_1, \dots, \mathscr{W}_\rho$, where $\mathscr{W}_0$ is the trivial bundle and where each $\mathscr{W}_i$ has
 rank $r_i=\rank(\mathscr{E}_i)$. 
 
 Having defined multigraded linear series to be certain quiver flag varieties, we now generalise Example~\ref{exa:flagqos} to show that every quiver flag variety is a multigraded linear series.

\begin{lemma}
\label{lem:tautologicalmultigraded} 
Let $\underline{\mathscr{E}} = (\mathscr{W}_0,\mathscr{W}_1,\dots,\mathscr{W}_\rho)$ denote the sequence comprising tautological bundles on a nonempty quiver flag variety $\mathcal{M}_\vartheta(Q,\underline{r})$. Then $\underline{\mathscr{E}}$ is weakly exceptional, the quiver of sections of $\underline{\mathscr{E}}$ is equal to $Q$, and the multigraded linear series is $\vert\underline{\mathscr{E}}\vert\cong\mathcal{M}_\vartheta(Q,\underline{r})$.
\end{lemma}
\begin{proof}
The nontrivial tautological bundles $\mathscr{W}_1,\dots,\mathscr{W}_\rho$ on  $\mathcal{M}_\vartheta(Q,\underline{r})$ are nonisomorphic since we assume that the inequalities \eqref{eqn:numerical} are strict (see Remark~\ref{rem:strict}). Each bundle $\mathscr{W}_i$  is globally generated by Corollary~\ref{coro:pullback}, and the sequence $\underline{\mathscr{E}}$ is strongly exceptional by Theorem~\ref{thm:quivertilt}, hence it is weakly exceptional. Since $r_i=\rank(\mathscr{W}_i)$ for $i\in Q_0$, it remains to prove that the quiver of sections of $\underline{\mathscr{E}}$ is the quiver $Q$. The quiver of sections of $\underline{\mathscr{E}}$ contains $Q$ as a subquiver, because a defining property of the functor $\mathfrak{M}_\vartheta(Q,\underline{r})$ from \eqref{eqn:functor} associates a morphism $\mathscr{W}_{\tail(a)}\to\mathscr{W}_{\head(a)}$ to each arrow $a\in Q_1$. If the quiver of sections of $\underline{\mathscr{E}}$ contains an arrow $a\not\in Q_1$, then there is an irreducible morphism $f\colon\mathscr{W}_{i}\to\mathscr{W}_{j}$ that does not arise from a composition of morphisms from \eqref{eqn:functor}. However, the composition of $f$ with any morphism $\mathscr{O}_X\to \mathscr{W}_i$ defines a section of $\mathscr{W}_j$, and the isomorphism from Corollary~\ref{coro:pullback}\two\ implies that sections of $\mathscr{W}_j$ factor via bundles $\mathscr{W}_i$ with $i\neq j$ only when the induced morphism $\mathscr{W}_i\to\mathscr{W}_j$ decomposes as a composition of morphisms $\mathscr{W}_{\tail(a)}\to\mathscr{W}_{\head(a)}$ arising from arrows in $Q$. This completes the proof.
\end{proof}
 
 \begin{theorem}
 \label{thm:morphism}
 The multigraded linear series $\vert\underline{\mathscr{E}}\vert$ of any weakly exceptional sequence of globally generated vector bundles $\underline{\mathscr{E}}=(\mathscr{O}_X,\mathscr{E}_1,\dots, \mathscr{E}_\rho)$ on a projective scheme $X$ is nonempty. Moreover, if $Q$ denotes the quiver of sections of $\underline{\mathscr{E}}$ then there is
 \begin{enumerate}
 \item[\one] a morphism of schemes 
  \[
  \varphi_{\vert\underline{\mathscr{E}}\vert} \colon X\longrightarrow
 \vert\underline{\mathscr{E}}\vert
  \]
 satisfying $\varphi_{\vert\underline{\mathscr{E}}\vert}^*(\mathscr{W}_i)=\mathscr{E}_i$ for all $0\leq
 i\leq \rho$; and
 \item[\two] a homomorphism of $\kk$-algebras 
  \[
  \Psi_{\vert\underline{\mathscr{E}}\vert}\colon\kk Q\longrightarrow
 \End\Big(\bigoplus_{i\in Q_0} \mathscr{E}_i\Big)
  \]
 which is surjective if
 and only if each $\mathscr{E}_i$ is simple, i.e., $\End(\mathscr{E}_i)=\kk$ for $i\geq 0$.
\end{enumerate} 
\end{theorem}
\begin{proof}
 For $x\in X$ and $i\in Q_0$, write $E_i:=\mathscr{E}_i\vert_x$ for the fibre of $\mathscr{E}_i$.  Since $\mathscr{E}_i$ is globally generated, the restriction of the epimorphism $H^0(\mathscr{E}_i)\otimes \mathscr{O}_X\rightarrow \mathscr{E}_i$ to the fibre over $x$ gives a surjective map of $\kk$-vector spaces $H^0(\mathscr{E}_i)\to E_i$. The first step is to show that the vector space $\bigoplus_{i\in Q_0}E_i$ carries the structure of a $\vartheta$-stable representation of $Q$.  To begin, choose sections $s_1,\dots, s_{r_\rho}\in H^0(\mathscr{E}_\rho)$ whose values $s_1(x),\dots, s_{r_\rho}(x)$ span the fibre $E_\rho$.  For $1\leq \alpha\leq r_\rho$, decompose where possible the map $s_\alpha\colon \mathscr{O}_X\to \mathscr{E}_\rho$ via the bundles $\mathscr{E}_i$ for $i<\rho$. Since $Q$ is the quiver of sections of $\underline{\mathscr{E}}$, each map in the resulting decomposition is of the form $f_a\colon \mathscr{E}_{\tail(a)}\to \mathscr{E}_{\head(a)}$ for some $a\in Q_1$.  Evaluate each map at $x\in X$, and define
    \begin{equation}
 \label{eqn:fibremaps}
 f_{(i)}:=\oplus f_a(x) \colon \bigoplus_{\{a\in Q_1 : \head(a)=i\}} E_{\tail(a)}\longrightarrow E_i
  \end{equation}
  for $i\geq 1$. The map $f_{(\rho)}$ is surjective since $s_1(x),\dots, s_{r_\rho}(x)$ span $E_\rho$, giving a matrix of full rank $w_\rho\in \Mat(r_\rho\times s_\rho,\kk)\cong \Rep(Q,\underline{r})_\rho$. Let $i<\rho$ be the largest index for which \eqref{eqn:fibremaps} is not surjective. Choose sections $\sigma_1,\dots, \sigma_{m}\in H^0(\mathscr{E}_i)$ whose values $\sigma_1(x),\dots, \sigma_{m}(x)$ span the cokernel of \eqref{eqn:fibremaps}. Decompose the maps $\sigma_\beta\colon \mathscr{O}_X\to \mathscr{E}_i$ as above to obtain new maps of the form $f_a\colon \mathscr{E}_{\tail(a)}\to \mathscr{E}_{\head(a)}$ for some $a\in Q_1$. Evaluate each map at $x\in X$, and add the resulting $\kk$-linear maps to those in \eqref{eqn:fibremaps} to create a direct sum with more summands. The new map $f_{(i)}$ is by construction surjective, giving a matrix of full rank $w_i\in \Rep(Q,\underline{r})_i$. Repeat until $f_{(i)}$ is surjective for all $i\geq 1$. These maps define a $\vartheta$-stable point $(w_i)\in \Rep(Q,\underline{r})$ by Lemma~\ref{lem:generic}, so $\vert\underline{\mathscr{E}}\vert = \mathcal{M}_\vartheta(Q,\underline{r})$  is nonempty.

 The construction of the $\vartheta$-stable representation $\bigoplus_{i\in Q_0}E_i$ extends to all fibres over a Zariski-open subset containing $x\in X$, so the locally free sheaf $\bigoplus_{i\in Q_0}\mathscr{E}_i$ on $X$ carries the structure of a family of $\vartheta$-stable representations of $Q$ of dimension vector $\underline{r}$. The universal properties of the fine moduli space $\mathcal{M}_\vartheta(Q,\underline{r})$ induce both $\varphi_{\vert\underline{\mathscr{E}}\vert}$ and $\Psi_{\vert\underline{\mathscr{E}}\vert}$. It remains to prove the statement on the surjectivity of $\Psi_{\vert\underline{\mathscr{E}}\vert}$.  For $i, j\in Q_0$, the map $\Psi_{\underline{\mathscr{E}}}$ sends the subspace $e_i(\kk Q)e_j$ generated by paths with tail at $i$ and head at $j$ to $\Hom(\mathscr{E}_i,\mathscr{E}_j)$. Since $Q$ is the quiver of sections of $\underline{\mathscr{E}}$, the restriction of $\Psi_{\underline{\mathscr{E}}}$ to $e_i(\kk Q)e_j$ maps onto $\Hom(\mathscr{E}_i,\mathscr{E}_j)$ for $i\neq j$. Since $Q$ is acyclic, we have $e_i(\kk Q)e_i = \kk e_i$ for $i\in Q_0$. Since $\Psi_{\underline{\mathscr{E}}}$ sends the idempotent $e_i\in \kk Q$ to the identity map on $\mathscr{E}_i$,  it follows that $\Psi_{\underline{\mathscr{E}}}$ is surjective if and only if $\End(\mathscr{E}_i)=\kk$ for all $i\in Q_0$.
\end{proof}

 \begin{example}
   Let $\mathscr{E}_1$ be a bundle of rank $r$ on a projective scheme $X$ and set $\underline{\mathscr{E}}:= (\mathscr{O}_X,\mathscr{E}_1)$. The morphism  $\varphi_{\vert\underline{\mathscr{E}}\vert}$ from Theorem~\ref{thm:morphism} coincides with the morphism $\varphi_{\vert\underline{\mathscr{E}}\vert}\colon X\longrightarrow \Gr(H^0(\mathscr{E}_1),r)$ to the linear series of higher rank that recovers $\mathscr{E}_1$ as the pullback of the tautological quotient bundle of rank $r$. In particular, if $\mathscr{E}_1$ has rank one then we recover the morphism to the classical linear series $\vert\underline{\mathscr{E}}\vert\cong \mathbb{P}^*(H^0(\mathscr{E}_1))$.
 \end{example}
   
 \begin{remark}
   If $X$ is a projective toric variety and each $\mathscr{E}_i$ has rank one then  $\vert\underline{\mathscr{E}}\vert$ coincides with
   the multilinear series introduced by Craw--Smith~\cite{CrawSmith}, in which case, Theorem~\ref{thm:morphism} specialises to \cite[Proposition~3.3, Corollary~4.2]{CrawSmith}.
    \end{remark}
 
  \begin{proposition}
  \label{prop:closedimmersion}
 Let $\underline{\mathscr{E}}=(\mathscr{O}_X,\mathscr{E}_1,\dots, \mathscr{E}_\rho)$ be a sequence of globally generated line bundles on a projective scheme $X$ and set $\mathscr{E}:=\bigotimes_{1\leq i\leq \rho}\mathscr{E}_i$. Assume that the multiplication map 
 \[
 H^0(\mathscr{E}_1)\otimes \dots \otimes H^0(\mathscr{E}_\rho)\longrightarrow H^0(\mathscr{E})
\]
 is surjective. Then $\varphi_{\vert\underline{\mathscr{E}}\vert}\colon
 X \longrightarrow \vert\underline{\mathscr{E}}\vert$ is a closed immersion if and only if $\mathscr{E}$ is very ample.
  \end{proposition}
 \begin{proof}
 The sequence $\underline{\mathscr{E}}$ is weakly exceptional since $X$ is projective. Let $\mathscr{W}_0,\mathscr{W}_1,\dots,\mathscr{W}_\rho$ denote the tautological line bundles on $\vert\underline{\mathscr{E}}\vert$. As a first step we construct a commutative diagram 
   \[
   \begin{CD}
   H^0(\mathscr{W}_1)\otimes \dots \otimes H^0(\mathscr{W}_\rho)@>>> H^0\Big(\!\textstyle{\bigotimes_{1\leq i\leq \rho} \mathscr{W}_i}\Big) \\
    @V{g}VV  @VV{h}V \\
   H^0(\mathscr{E}_1)\otimes \dots \otimes H^0(\mathscr{E}_\rho) @>>>   H^0(\mathscr{E})
   \end{CD}
 \] 
  where the horizontal maps are given by multiplication. The proof of Theorem~\ref{thm:morphism} shows that the sheaf $\bigoplus_{0\leq i\leq \rho}\mathscr{E}_i$ carries the structure of a $\vartheta$-stable representation of the quiver of sections $Q$ and, moreover, the assignment sending $a\in Q$ to the corresponding morphism $f_a \colon \mathscr{E}_{\tail(a)}\to\mathscr{E}_{\head(a)}$ extends to a surjective map of $\kk$-vector spaces $e_{i}(\kk Q)e_{j}\to \Hom(\mathscr{E}_{i},\mathscr{E}_{j})$ for $i<j$. For $i=0$, the composition of this map with the isomorphism from Corollary~\ref{coro:pullback} gives a surjective map 
 \begin{equation}
 \label{eqn:surjectiongj}
 g_j\colon H^0(\mathscr{W}_j)\longrightarrow H^0(\mathscr{E}_j) \text{ for }j\geq 1.
 \end{equation}
The tensor product of all such maps $g:=\otimes_j g_j$ provides one vertical map in the diagram. More generally, for any $\theta=(\theta_1,\dots,\theta_\rho)\in \ZZ^\rho$, identify $H^0(\mathscr{W}_1^{\theta_1}\otimes\dots\otimes\mathscr{W}_\rho^{\theta_\rho})$ with the subspace of the Cox ring $\kk[y_a : a\in Q]$ of the toric variety $\vert\underline{\mathscr{E}}\vert$ spanned by monomials of degree $\theta$, and define
\begin{equation}
\label{eqn:gtheta}
g_\theta\colon H^0\big(\mathscr{W}_1^{\theta_1}\otimes\dots\otimes\mathscr{W}_\rho^{\theta_\rho}\big)\longrightarrow H^0(\mathscr{E}_1^{\theta_1}\otimes\dots\otimes\mathscr{E}_\rho^{\theta_\rho})
\end{equation}
by sending the monomial $\prod_{a\in Q_1}y_a^{m_a}$ to the section $\prod_{a\in Q_1}f_a^{m_a}$, where for each $m_a\in \NN$ we have $f_a^{m_a}\in H^0(\mathscr{E}_{\head(a)}^{m_a}\otimes\mathscr{E}_{\tail(a)}^{-m_a})$.  If we define the right hand vertical map of the diagram to be $h:=g_\theta$ for $\theta=(1,\dots,1)$, then the diagram commutes. The left-hand vertical map is surjective, and the lower horizontal map is surjective by assumption, so the right hand vertical map must also be surjective. The induced closed immersion of projective spaces provides the right hand vertical map in the following commutative diagram of projective varieties
   \[
   \begin{CD}
   \vert\underline{\mathscr{E}}\vert @>>> \mathbb{P}^*\Big(H^0\big(\bigotimes_{1\leq i\leq \rho} \mathscr{W}_i\big)\Big)\\
    @A{\varphi_{\vert\underline{\mathscr{E}}\vert}}AA  @AAA \\
   X @>>>  \mathbb{P}^*\big(H^0(\mathscr{E})\big) 
   \end{CD}
 \] 
  where the horizontal maps are the morphisms to the classical linear series. Since $\vert\underline{\mathscr{E}}\vert$ is a smooth toric variety, the ample bundle $\bigotimes_{1\leq i\leq \rho} \mathscr{W}_i$ is very ample, so the morphism on the top row is a closed immersion. It follows that $\varphi_{\vert\underline{\mathscr{E}}\vert}$ is a closed immersion if and only if the lower horizontal morphism is a closed immersion, which holds if and only if $\mathscr{E}$ is very ample.
 \end{proof}

\begin{example}
For $\underline{\mathscr{E}}=(\mathscr{O}_X,\mathscr{E}_1)$ with $\rank(\mathscr{E}_1)=1$, the morphism $\varphi_{\vert\underline{\mathscr{E}}\vert}$ is a closed immersion if and only if $\mathscr{E}_1$ is very ample.
\end{example}
 
 \begin{example}
 \label{exa:F2}
 The Hirzebruch surface $\mathbb{F}_2 =
 \mathbb{P}_{\mathbb{P}^1}\big(\mathscr{O}\oplus \mathscr{O}(2)\big)$
 is a toric variety with fan $\Sigma$ shown in Figure~\ref{fig:F2}(a)
 and Cox ring $\kk[x_1, x_2, x_3, x_4]$. Let $\mathscr{E}_1$ and $\mathscr{E}_2$ denote
 the line bundles on $\mathbb{F}_2$ that induce the morphisms to $\mathbb{P}^1$
 and $\mathbb{P}(1,1,2)$ respectively, and write
 $\underline{\mathscr{E}}:=(\mathscr{O}_{\mathbb{F}_2},\mathscr{E}_1,\mathscr{E}_2)$.  Since $\mathbb{F}_2$ is toric, the
 quiver of sections $Q$ can be described simply in terms of the
 torus-invariant bases for the spaces $\Hom(\mathscr{E}_i,\mathscr{E}_j)$ or, equivalently, monomials from the Cox ring. The resulting quiver, illustrated  in Figure~\ref{fig:F2}(b) 
with monomials indicating bases for $\Hom$-spaces, coincides with that from
 Figure~\ref{fig:flag}(b), so the multigraded linear series is
 $\vert\underline{\mathscr{E}}\vert =
 \mathbb{P}_{\mathbb{P}^1}\big(\mathscr{O}\oplus\mathscr{O}(1)\oplus\mathscr{O}(1)\big)$.
 In this case, $\Psi_{\underline{\mathscr{E}}}$ is
 surjective with kernel the two-sided ideal $\langle
 a_1a_5-a_2a_4\rangle$, and $\varphi_{\vert\underline{\mathscr{E}}\vert} \colon \mathbb{F}_2\to
 \vert\underline{\mathscr{E}}\vert$ is a closed immersion with image $\mathbb{V}(y_1y_5-y_2y_4)\git_\vartheta G$ (compare \cite[Example~4.12]{CrawSmith}).
 \begin{figure}[!ht]
    \centering
    \mbox{
   \subfigure[]{
        \psset{unit=1cm}
        \begin{pspicture}(-0.5,-0.1)(2.5,3.0)
          \pspolygon[linecolor=white,fillstyle=hlines*,
          hatchangle=25](1,1)(1,3)(2,3)(2,1)
          \pspolygon[linecolor=white,fillstyle=hlines*,
          hatchangle=5](1,1)(1,3)(0,3)
          \pspolygon[linecolor=white,fillstyle=hlines*,
          hatchangle=55](1,1)(0,3)(0,0)(1,0)
          \pspolygon[linecolor=white,fillstyle=hlines*,
          hatchangle=165](1,1)(2,1)(2,0)(1,0)
          \rput(2,0.7){\psframebox*{$1$}}
          \rput(1.3,2){\psframebox*{$2$}}
          \rput(-0.3,3){\psframebox*{$3$}}
          \rput(0.7,0){\psframebox*{$4$}}
          \cnode*[fillcolor=black](0,0){3pt}{P1}
          \cnode*[fillcolor=black](1,0){3pt}{P2}
          \cnode*[fillcolor=black](2,0){3pt}{P3}
          \cnode*[fillcolor=black](0,1){3pt}{P4}
          \cnode*[fillcolor=black](1,1){3pt}{P5}
          \cnode*[fillcolor=black](2,1){3pt}{P6}
          \cnode*[fillcolor=black](0,2){3pt}{P7}
          \cnode*[fillcolor=black](1,2){3pt}{P8}
          \cnode*[fillcolor=black](2,2){3pt}{P9}
          \cnode*[fillcolor=black](0,3){3pt}{P10}
          \cnode*[fillcolor=black](1,3){3pt}{P11}
          \cnode*[fillcolor=black](2,3){3pt}{P12}
          \ncline[linewidth=2pt]{->}{P5}{P10}
          \ncline{-}{P5}{P11}
          \ncline[linewidth=2pt]{->}{P5}{P6}
          \ncline[linewidth=2pt]{->}{P5}{P2}
          \ncline[linewidth=2pt]{->}{P5}{P8}
        \end{pspicture}}
      \qquad \qquad
      \subfigure[]{
        \psset{unit=1cm}
        \begin{pspicture}(-0.5,-0.6)(3,2)
          \cnodeput(0.5,0){A}{1}
          \cnodeput(2.5,0){B}{1} 
          \cnodeput(0.5,2){C}{1}
          \psset{nodesep=0pt}
          \ncline[offset=5pt]{->}{A}{B}\lput*{:U}{\footnotesize{$x_1$}}
          \ncline[offset=-5pt]{->}{A}{B}\lput*{:U}{\footnotesize{$x_3$}}
          \ncline{->}{A}{C} \lput*{:270}{\footnotesize{$x_4$}}
          \ncline[offset=5pt]{<-}{C}{B}\lput*{:0}{\footnotesize{$x_2x_3$}}
          \ncline[offset=5pt]{->}{B}{C}\lput*{:180}{\footnotesize{$x_1x_2$}}
      \end{pspicture}}
    }
    \caption{(a) the fan of $\mathbb{F}_2$; (b) a quiver of sections on $\mathbb{F}_2$\label{fig:F2}}
  \end{figure}
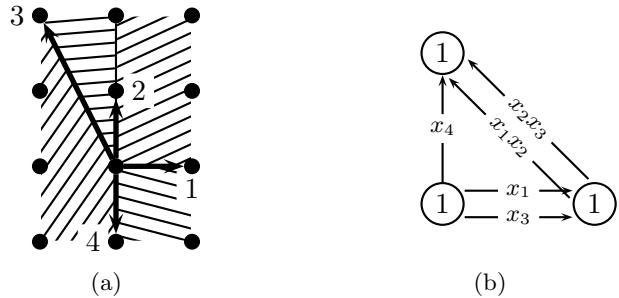
 \end{example}
 
 \begin{remark}
   As Example~\ref{exa:F2} illustrates, the homomorphism $\Psi_{\underline{\mathscr{E}}}$ is
   rarely injective because the maps $\mathscr{E}_i\to \mathscr{E}_j$ that determine the
   arrows in $Q$ may factor in more than one way through bundles $\mathscr{E}_k$
   with $i<k<j$. One expects nevertheless that the kernel plays an
   important geometric role, at least when each $\mathscr{E}_i$ is simple. For example, if
   $X$ is a toric variety and if each $\mathscr{E}_i$ has rank one, the main
   result of Craw--Smith~\cite{CrawSmith} states that, after augmenting any list of globally generated line bundles by up to two additional bundles, the 
   two-sided ideal $\Ker(\Psi_{\underline{\mathscr{E}}})$ in the noncommutative algebra $\kk
   Q$ determines explicitly the image of the
   morphism $\varphi_{\vert\underline{\mathscr{E}}\vert} \colon X\to
   \vert\underline{\mathscr{E}}\vert$, in which case $X$ is the
   fine moduli space of $\vartheta$-stable 
   $\End\big(\bigoplus_{0\leq i\leq \rho} \mathscr{E}_i\big)$-modules.
  \end{remark}

\section{Multigraded Pl\"{u}cker embedding and the Cox ring}
As an application of multigraded linear series we present a natural
 ambient space for quiver flag varieties, generalising the Pl\"{u}cker
 embedding of a Grassmannian.   
 
 For any quiver flag variety $\mathcal{M}_\vartheta=\mathcal{M}_\vartheta(Q,\underline{r})$, the determinants $\det(\mathscr{W}_1),\dots,
 \det(\mathscr{W}_\rho)$ of the nontrival tautological bundles are distinct, globally generated line bundles, and the sequence 
 \begin{equation}
 \label{eqn:detbundles}
 \underline{\det}(\mathscr{W}):=  \big(\mathscr{O}_{\mathcal{M}_\vartheta},
 \det(\mathscr{W}_1),\dots, \det(\mathscr{W}_\rho)\big)
 \end{equation}
 defines a quiver of sections $Q^\prime$. For
 $\underline{r}^\prime = (1,\dots, 1)$ and $\vartheta^\prime =
 (-\rho,1,\dots, 1)$, the multigraded linear series
 $\vert\underline{\det}(\mathscr{W})\vert =
 \mathcal{M}_{\vartheta^\prime}(Q^\prime,\underline{r}^\prime)$ is a
 smooth projective toric variety that carries tautological line bundles $\mathscr{O}_{\vert\underline{\det}(\mathscr{W})\vert}, \mathscr{W}^\prime_1,\dots,\mathscr{W}^\prime_\rho$.  Theorem~\ref{thm:morphism} constructs a morphism 
  \[
  \varphi_{\vert\underline{\det}(\mathscr{W})\vert}\colon
 \mathcal{M}_\vartheta \longrightarrow \vert\underline{\det}(\mathscr{W})\vert
 \]
 satisfying
 $\varphi_{\vert\underline{\det}(\mathscr{W})\vert}^*(\mathscr{W}^\prime_i)=\det(\mathscr{W}_i)$
 for $i\geq 0$. We call this the \emph{multigraded Pl\"{u}cker embedding}. To justify the terminology we establish the following result.

 \begin{proposition}
 \label{prop:Pluckerembedding}
 The multigraded Pl\"{u}cker morphism $\varphi_{\vert\underline{\det}(\mathscr{W})\vert}$ is a closed immersion.
 \end{proposition}
 \begin{proof}
 We proceed by induction using the tower \eqref{eqn:tower}. The case $\rho=1$ gives $\mathcal{M}_\vartheta\cong \Gr(\kk^{s_1},r_1)$. The quiver of sections $Q^\prime$, which has $Q^\prime_0=\{0,1\}$ and arrow set with $h^0(\det(\mathscr{W}_1))=\binom{s_1}{r_1}$ arrows from 0 to 1, defines the classical linear series $\vert\underline{\det}(\mathscr{W})\vert\cong\mathbb{P}^*(\bigwedge^{r_1}\kk^{s_1})$ and the multigraded morphism $\varphi_{\vert\underline{\det}(\mathscr{W})\vert}$ is the Pl\"{u}cker embedding.   Assume that the result holds for $Y_{\rho-1}$ and relabel the tautological bundles on $Y_{\rho-1}$ as $\mathscr{V}_i:= \mathscr{W}^{(\rho-1)}_i$ for $i\leq \rho-1$.  The sequence 
 \[
 \underline{\det}(\mathscr{V}):=  \big(\mathscr{O}_{Y_{\rho-1}},
 \det(\mathscr{V}_1),\dots, \det(\mathscr{V}_{\rho-1})\big)
 \] defines the closed immersion $\varphi_{\vert\underline{\det}(\mathscr{V})\vert}\colon
 Y_{\rho-1}\rightarrow \vert\underline{\det}(\mathscr{V})\vert$ with $\varphi_{\vert\underline{\det}(\mathscr{V})\vert}^*(\mathscr{V}^\prime_i)=\det(\mathscr{V}_i)$ for each tautological bundle $\mathscr{V}^\prime_i$ on $\vert\underline{\det}(\mathscr{V})\vert$.  Applying Theorem~\ref{thm:tower} gives $\vert\underline{\det}(\mathscr{W})\vert \cong\mathbb{P}^*(\mathscr{F}^\prime_\rho)$ for the bundle $\mathscr{F}_\rho^\prime = \bigoplus_{\{a'\in Q_1^\prime : \head(a^\prime)=\rho\}} \mathscr{V}^\prime_{\tail(a^\prime)}$ on $\vert\underline{\det}(\mathscr{V})\vert$.  We now construct a commutative diagram
   \begin{equation}
   \label{eqn:Pluckerdiagram}
   \begin{CD}
   \mathcal{M}_\vartheta\cong\Gr(\mathscr{F}_\rho,r_\rho) @>>> \mathbb{P}^*\big(\bigwedge^{r_\rho}\mathscr{F}_\rho\big)@>>> \mathbb{P}^*\big(\mathscr{G}_\rho\big) @>>> \vert\underline{\det}(\mathscr{W})\vert \cong \mathbb{P}^*(\mathscr{F}^\prime_\rho)\\
    @V{\pi}VV  @VVV @VVV @VV{\pi^\prime}V\\
   Y_{\rho-1} @=  Y_{\rho-1} @=  Y_{\rho-1} @>{\varphi_{\vert\underline{\det}(\mathscr{V})\vert}}>>\vert\underline{\det}(\mathscr{V})\vert
   \end{CD}
 \end{equation}
 for $\mathscr{G}_\rho:=\bigoplus_{\{a^\prime\in Q_1^\prime: \head(a^\prime = \rho\}} \det(\mathscr{V}_{\tail(a^\prime)})$, where each morphism on the top row is a closed immersion. The left hand square is the relative Pl\"{u}cker morphism $\Gr(\mathscr{F}_\rho,r_\rho)\to \mathbb{P}^*(\bigwedge^{r_\rho}\mathscr{F}_\rho)$ over $Y_{\rho-1}$, so the left horizontal map is a closed immersion by \cite[Proposition~9.8.4]{EGA}. The right hand square is the base change to $Y_{\rho-1}$ of the projective bundle $\pi^\prime$, and the closed immersion on the top row is induced by the epimorphism $\mathscr{F}_\rho^\prime \to \varphi_{\vert\underline{\det}(\mathscr{V})\vert}^*(\mathscr{F}^\prime_\rho)=\mathscr{G}_\rho$. To construct the closed immersion in the central square, we define an epimorphism $\mathscr{G}_\rho\to \bigwedge^{r_\rho}\mathscr{F}_\rho$.   
 
 The epimorphism $H^0(\mathscr{W}_\rho)\otimes\mathscr{O}\to\mathscr{W}_\rho$ constructed in Corollary~\ref{coro:pullback} factors via epimorphisms
  \[
 H^0(\mathscr{W}_\rho)\otimes\mathscr{O} \to\cdots \to  \bigoplus_{j\leq i}\mathscr{W}_j^{\oplus d_{i,j}}\to \cdots \to\pi^*(\mathscr{F}_\rho)\to \mathscr{W}_\rho,
  \]
 for $\pi^*(\mathscr{F}_\rho):= \bigoplus_{\{a\in Q_1 : \head(a)=\rho\}} \mathscr{W}_{\tail(a)}$, where $d_{i,j}$ is the number of paths in $Q$ with tail at $j$, head at $\rho$ that do not factor via any $k\in Q_0$ satisfying $j<k\leq i$. Applying $\bigwedge^{r_\rho}(-)$ gives epimorphisms
 \[
 H^0\big(\!\det(\mathscr{W}_\rho)\big)\otimes\mathscr{O} \rightarrow \cdots \to\bigwedge^{r_\rho} \bigoplus_{j\leq i}\mathscr{W}_j^{\oplus d_{i,j}} \to\cdots \to \bigwedge^{r_\rho} \pi^*(\mathscr{F}_\rho)\to \det(\mathscr{W}_\rho)
 \]
 When $\det(\mathscr{W}_j)$ occurs as a summand of a term of this sequence, the corresponding section of $\det(\mathscr{W}_\rho)$ factors via $\det(\mathscr{W}_j)$.  Since $Q^\prime$ is the quiver of sections of $\underline{\det}(\mathscr{W})$,  arrows $a^\prime\in Q_1^\prime$ with $\head(a^\prime)=\rho$ exist precisely when sections of $\det(\mathscr{W}_\rho)$ which factor via $\det(\mathscr{W}_{\tail(a^\prime)})$, so the direct sum of all such bundles $\pi^*(\mathscr{G}_\rho)=\bigoplus_{\{a'\in Q_1^\prime : \head(a^\prime)=\rho\}} \det(\mathscr{W}_{\tail(a^\prime)})$ fits in to a sequence of epimorphisms
 \[
 H^0\big(\!\det(\mathscr{W}_\rho)\big)\otimes\mathscr{O} \to\pi^*(\mathscr{G}_\rho)\to \textstyle{\pi^*\big(\bigwedge^{r_\rho}\mathscr{F}_\rho\big)}\to \det(\mathscr{W}_\rho).
 \]
 The second map is the pullback to $\mathcal{M}_\vartheta$ of an  epimorphism $\mathscr{G}_\rho\to \bigwedge^{r_\rho}\mathscr{F}_\rho$ on $Y_{\rho-1}$, so the morphism $\tau\colon \mathcal{M}_\vartheta\to\vert\underline{\det}(\mathscr{W})\vert$ on the top row of \eqref{eqn:Pluckerdiagram} is a closed immersion. 
 It remains to show that $\tau = \varphi_{\vert\underline{\det}(\mathscr{W})\vert}$.  Induction and commutativity of \eqref{eqn:Pluckerdiagram} give $\tau^*(\mathscr{W}^\prime_i)=\det(\mathscr{W}_i)$ for $i<\rho$. The bundle $\mathscr{W}^\prime_\rho = \mathscr{O}_{\mathbb{P}^*(\mathscr{F}^\prime_\rho)}(1)$ pulls back to $\mathscr{O}_{\mathbb{P}^*(\bigwedge^{r_\rho}\mathscr{F}_\rho)}(1)$ by \cite[(4.1.2.1)]{EGA2}, which in turn pulls back via the relative Pl\"{u}cker morphism to $\det(\mathscr{W}_\rho)$, giving $\tau^*(\mathscr{W}^\prime_\rho)=\det(\mathscr{W}_\rho)$ as required.
 \end{proof}

\begin{example}
Consider the quiver $Q$ from Figure~\ref{fig:flag}(a) that defines $\mathcal{M}_\vartheta\cong \Fl(n; r_1,\dots,r_\rho)$. For each $1\leq i\leq \rho$, none of the $\binom{n}{r_i}$
 sections $\mathscr{O}_{\mathcal{M}_\vartheta}\to \det(\mathscr{W}_i)$ factors via $\det(\mathscr{W}_j)$ with $j\neq i$, so the quiver of sections $Q^\prime$ has vertex set $Q^\prime_0=\{0,1,\dots\rho\}$ and arrow set with $h^0(\det(\mathscr{W}_i))=\binom{n}{r_i}$ arrows from 0 to $i$. The multigraded linear series is the product $\vert\underline{\det}(\mathscr{W})\vert\cong\prod_{i=1}^{\rho} \mathbb{P}^*(\bigwedge^{r_i}\kk^{n})$, and the closed immersion $\varphi_{\vert\underline{\det}(\mathscr{W})\vert}\colon\mathcal{M}_\vartheta \longrightarrow \prod_{i=1}^{\rho}\mathbb{P}^*(\bigwedge^{r_i}\kk^{n})$ is classical; see Fulton~\cite[\S 9.1]{Fulton}. 
\end{example}
  
  \begin{example}
For the quiver $Q$ from Figure~\ref{fig:Plucker}(a) and dimension vector $\underline{r}=(1,2,2)$,  the quiver flag variety $\mathcal{M}_\vartheta\cong \Gr(\mathscr{O}\oplus\mathscr{W}_1^{\oplus 2},2)$ is a $\Gr(\kk^5,2)$-bundle over $Y_1=\Gr(\kk^4,2)$.  Lemma~\ref{lem:tautologicalmultigraded} shows that $Q$ is the quiver of sections of $(\mathscr{O},\mathscr{W}_1,\mathscr{W}_2)$. Counting paths in $Q$ and applying Corollary~\ref{coro:pullback} gives epimorphisms $\mathscr{O}^{\oplus 4}\rightarrow \mathscr{W}_1$ and $\mathscr{O}^{\oplus 9}\rightarrow \mathscr{W}_2$, where the latter factors through $\mathscr{F}_2=\mathscr{O}\oplus\mathscr{W}_1^{\oplus 2}$. 
\begin{figure}[!ht]
    \centering
    \mbox{
   \subfigure[]{
        \psset{unit=1cm}
        \begin{pspicture}(0,-0.5)(3,2.4)
       \cnodeput(0.5,0){A}{1}
          \cnodeput(2.5,0){B}{2} 
          \cnodeput(0.5,2){C}{2}
          \psset{nodesep=0pt}
          \ncline[offset=4.5pt]{->}{A}{B}
          \ncline[offset=-4.5pt]{->}{A}{B}
           \ncline[offset=1.5pt]{->}{A}{B}
          \ncline[offset=-1.5pt]{->}{A}{B}
          \ncline{->}{A}{C}
          \ncline[offset=2pt]{<-}{C}{B}
          \ncline[offset=2pt]{->}{B}{C}
        \end{pspicture}}
      \qquad \qquad
      \subfigure[]{
        \psset{unit=1cm}
        \begin{pspicture}(0,-0.5)(3,2)
          \cnodeput(0.5,0){A}{1}
          \cnodeput(2.5,0){B}{1} 
          \cnodeput(0.5,2){C}{1}
          \psset{nodesep=0pt}
          \ncline{->}{A}{B}\lput*{:U}{\footnotesize{$6$}}
          \ncline{->}{A}{C} \lput*{:270}{\footnotesize{24}}
          \ncline{->}{B}{C}\lput*{:225}{\footnotesize{2}}
      \end{pspicture}}
    }
    \caption{(a) The pair $(Q,\underline{r})$ giving $\mathcal{M}_\vartheta$; (b) the pair $(Q^\prime,\underline{r}^\prime)$ giving $\vert\underline{\det}(\mathscr{W})\vert$. In (b), the numbers labeling arrows indicate the number of arrows in $Q^\prime$.\label{fig:Plucker}}
  \end{figure}
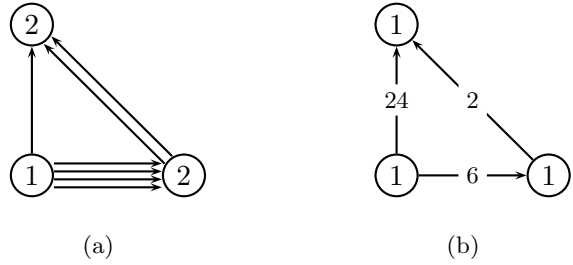
  Both $\mathscr{W}_1$ and $\mathscr{W}_2$ have rank 2, so there exist epimorphisms $\mathscr{O}^{\oplus 6}\rightarrow \det(\mathscr{W}_1)$ and $\mathscr{O}^{\oplus 36}\rightarrow \det(\mathscr{W}_2)$, where the latter factors through $\bigwedge^2\mathscr{F}_2=\det(\mathscr{W}_1)^{\oplus 2}\oplus\mathscr{W}_1^{\oplus 2}\oplus (\mathscr{W}_1\otimes\mathscr{W}_1)$. Since $h^0(\det(\mathscr{W}_1))=6$, exactly 12 of the 36 sections of $\det(\mathscr{W}_2)$ decompose via $\det(\mathscr{W}_1)$, giving the quiver of sections $Q^\prime$ from Figure~\ref{fig:Plucker}(b). Theorem~\ref{thm:tower} gives $\vert\underline{\det}(\mathscr{W})\vert=\mathbb{P}_{\mathbb{P}^5}(\mathscr{O}^{\oplus 24}\oplus \mathscr{O}(1)^{\oplus 2})$, so $\varphi_{\vert\underline{\det}(\mathscr{W})\vert}$ embeds $\mathcal{M}_\vartheta$ as a subvariety of codimension 20.
  \end{example} 

To conclude we formulate a conjecture linking the multigraded Pl\"{u}cker morphism of a quiver flag variety with its Cox ring. In light of Lemma~\ref{lem:nefcone}, the Cox ring of $\mathcal{M}_\vartheta$ is the $\kk$-algebra
 \[
 \Cox(\mathcal{M}_\vartheta):= \bigoplus_{(\theta_1,\dots,
   \theta_\rho)\in
   \ZZ^\rho}H^0\Big(\mathcal{M}_\vartheta,\det(\mathscr{W}_1)^{\theta_1}\otimes
 \dots \otimes \det(\mathscr{W}_\rho)^{\theta_\rho}\Big).
 \]
 Since $\mathcal{M}_\vartheta$ is a Mori Dream Space, Hu--Keel~\cite[Proposition~2.9]{HuKeel} shows that $\Cox(\mathcal{M}_\vartheta)$ is finitely generated and that $\mathcal{M}_\vartheta$ is isomorphic to $\Spec\big(\!\Cox(\mathcal{M}_\vartheta)\big)\git_{\vartheta^\prime} \Hom(\ZZ^\rho,\kk^\times)$ for the ample linearisation $\vartheta^\prime = (-\rho,1,\dots,1)$. For example,  the Cox ring of the toric variety $\vert\underline{\det}(\mathscr{W})\vert=\mathcal{M}_{\vartheta^\prime}(Q^\prime,\underline{r}^\prime)$ arising from the sequence $\underline{\det}(\mathscr{W})$ is the $\ZZ^{\rho}$-graded polynomial ring $\kk[y_a : a\in Q^\prime]$, and the GIT construction is the defining construction $\Rep(Q^\prime,\underline{r}^\prime)\git_{\vartheta^\prime}G(\underline{r}^\prime)$. 
 
 If $\mathscr{W}_1^\prime,\dots,\mathscr{W}^\prime_\rho$ denote the nontrivial tautological bundles on $\vert\underline{\det}(\mathscr{W})\vert$, then the maps \eqref{eqn:gtheta} for the multigraded Pl\"{u}cker morphism are
 \[
g_\theta\colon H^0\big((\mathscr{W}^\prime_1)^{\theta_1}\otimes\dots\otimes(\mathscr{W}_\rho^\prime)^{\theta_\rho}\big)\longrightarrow H^0\big(\det(\mathscr{W}_1)^{\theta_1}\otimes\dots\otimes\det(\mathscr{W}_\rho)^{\theta_\rho}\big)
\]
for $\theta\in \ZZ^\rho$. These maps are surjective as $\theta$ ranges over each standard basis vector of $\ZZ^\rho$ by \eqref{eqn:surjectiongj}, but they need not a priori be surjective otherwise. Nevertheless we propose the following:
 
 \begin{conjecture}
 \label{conj:Plucker}
 The homomorphism of $\ZZ^\rho$-graded $\kk$-algebras
  \begin{equation}
  \label{eqn:Coxgens}
  \bigoplus_{\theta\in \ZZ^\rho} g_\theta\colon \kk[y_a : a\in Q^\prime_1] \longrightarrow \Cox(\mathcal{M}_\vartheta)
  \end{equation}
is surjective. 
 \end{conjecture}
 
  The multigraded Pl\"{u}cker embedding can be recovered from \eqref{eqn:Coxgens} by applying $\Spec(-)$ and taking the GIT quotients of the $\Hom(\ZZ^\rho,\kk^*)$-action linearised by $\vartheta^\prime\in \ZZ^\rho$. In particular, the homogeneous ideal $I:=\Ker(\bigoplus_{\theta\in \ZZ^\rho} g_\theta)$ cuts out the image of the multigraded Pl\"{u}cker embedding, in the sense that $\mathcal{M}_\vartheta\cong \mathbb{V}(I)\git_{\vartheta^\prime}\Hom(\ZZ^\rho,\kk^\times)$. It would be interesting to find a natural set of generators for $I$ and hence describe explicitly the Cox ring of $\mathcal{M}_\vartheta$.

 \begin{remark}
 \begin{enumerate}
  \item[\one] One might hope to prove the conjecture by working with the Cox rings of varieties on the top row of \eqref{eqn:Pluckerdiagram}. However, the Cox ring of $\mathbb{P}^*\big(\bigwedge^{r_\rho}\mathscr{F}_\rho\big)$ is not a priori finitely generated; even for the projectivisation of a vector bundle on a toric variety, this question was posed by Hering--Musta{\c{t}}{\v{a}}--Payne~\cite[Question~8.2]{HMP}.
  \item[\two]   Given the conjecture,  Hu--Keel~\cite[Proposition~2.11]{HuKeel} implies that all morphisms obtained by running the Mori programme on $\mathcal{M}_\vartheta$ are obtained by restriction from those of the ambient toric variety $\vert\underline{\det}(\mathscr{W})\vert$.  
  \end{enumerate}
 \end{remark}
 
 \def\cprime{$'$}

 
 \end{document}